\documentclass[11pt, a4paper]{article}                     
\usepackage{graphicx}
\usepackage[english]{babel}
\usepackage{amsmath,amssymb, amsthm}
\usepackage[left=3.3cm, right=3.3cm, bottom=4cm]{geometry} 
\usepackage{mathptmx}      
\theoremstyle{plain}
\newtheorem{lemma}{Lemma}
\newtheorem{proposition}{Proposition}
\newtheorem{corollary}{Corollary}
\newtheorem{theorem}{Theorem}
\theoremstyle{definition}

\newtheorem{definition}{Definition}
\newtheorem{remark}{Remark}

\newcommand{\new}{\newcommand}
\new{\set}[1]{\{{#1}\}}
\new{\rone}{{\mathbb R}}
\new{\nat}{{\mathbb N}}
\newcommand{\argmin}{\operatornamewithlimits{argmin}}
\new{\prox}{\operatorname{prox}}
\new{\proxj}{\prox_J^{H(x_n)}}
\new{\Prob}[1]{\mathrm{Pr}\left(\, #1 \right)}
\new{\E}{{\mathrm E}}
\new{\eps}{\varepsilon}
\new{\nor}[1]{\|{#1}\|}
\new{\norr}[1]{\left\|{#1}\right\|_\rho}
\new{\scal}[2]{\langle{#1},{#2}\rangle}
\new{\eqn}[1]{~(\ref{#1})}
\new{\norh}[1]{\left\|{#1}\right\|_\hh}
\new{\noru}[1]{|\!|\!|{#1}|\!|\!|}
\new{\scalh}[2]{\langle{#1},{#2}\rangle_{\hh}}
\new{\scaldue}[2]{\left\langle{#1},{#2}\right\rangle_{\rho}}
\new{\bg}{\begin}

\new{\beeq}[2]{\begin{equation}\label{#1}{#2}\end{equation}}
\new{\room}{\ \ \ \ }

\new{\marg}{\rho_X}
\new{\prob}{\rho}
\new{\dm}{d}
\new{\dd}{D}
\new{\ts}{\vz}
\new{\n}{n}
\new{\vz}{{\mathbf z}}
\new{\vx}{{\mathbf x}}
\new{\vy}{{\mathbf y}}
\new{\vs}{{\mathbf s}}
\new{\vv}{{\mathbf v}}
\new{\kk}{K}
\new{\esp}[1]{{\mathcal E}({#1})}
\new{\emp}[1]{{\mathcal E_\vz}({#1})}
\new{\hh}{{\mathcal H}}
\new{\ldue}{L^2(X,\marg)}
\new{\lp}{L^p(X,\marg)}
\new{\la}{\lambda}
\new{\pone}{\tau}
\new{\ptwo}{\mu}
\new{\tah}{f_\hh}
\new{\regr}{f_\rho}
\new{\fz}{f_\ts}
\new{\fzl}{f_{\ts}^\la}
\new{\phiv}{{\mathcal E}_{\pone}}
\new{\F}{F}
\new{\ff}{f^*}

\new{\som}[2]{\sum_{#1=1}^{#2}}
\new{\K}{{\mathrm K}}
\new{\Z}{{\mathrm Z}}
\new{\LL}{{\mathrm L}}
\new{\St}[1]{\mathbf{S}_{#1}}
\new{\st}[1]{S_{#1}}
\new{\J}{J}
\new{\tv}{\mathcal J}
\new{\FF}{\mathcal F}
\new{\dv}{\textrm{div}}
\new{\w}{\beta}
\new{\vw}{{\mathbf \w}}
\new{\one}{{\mathbf 1}}
\new{\vf}{\textrm{f}}
\new{\vg}{\textrm{g}}
\new{\dom}{X}
\new{\al}{\alpha}
\new{\be}{\beta}
\new{\id}{I}
\new{\TT}{\mathcal T_\sigma}
\new{\yy}{y}
\new{\B}{B_\sigma}

\new{\jj}{j}
\new{\jjj}{j'}
\new{\jjjj}{\gamma}
\new{\ii}{i}
\new{\iii}{i'}
\new{\iiii}{k}
\new{\itext}{p}
\new{\itint}{q}
\new{\xii}{x_\ii}
\new{\yi}{y_\ii}
\new{\dej}{\partial_\jj}
\new{\step}{\sigma}
\new{\GG}{{\mathcal G}}
\new{\Gk}{{\mathcal G}_k}
\new{\PP}{P}
\new{\kernel}{k}
\new{\ud}{\,\mathrm{d}}

\begin{document}

\title{Convergence analysis of a proximal Gauss-Newton method}

\author{Saverio Salzo$^1$\and Silvia Villa$^2$}
\maketitle

\begin{center}
$^1$ DISI, Universit\`a di Genova, \\
Via Dodecaneso 35, 16146 Genova, Italy, salzo@disi.unige.it\\
$^2$ DIMA \& DISI, Universit\`a di Genova, \\
Via Dodecaneso 35, 16146 Genova, Italy, villa@dima.unige.it
\end{center}

\begin{abstract} An extension of the Gauss-Newton algorithm is proposed to find local minimizers of penalized nonlinear least squares problems, under generalized Lipschitz assumptions.  Convergence results of local type are obtained, as well as an estimate of the radius of the convergence ball.
Some applications for solving constrained nonlinear  equations are discussed and  the numerical performance of the method is assessed on some significant test problems.\\


\end{abstract}

\section{Introduction}

Given Hilbert  spaces $X$ and $Y$, a Fr\'echet differentiable nonlinear operator $F:X\rightarrow Y$,  and a convex lower semicontinuous penalty functional $J:X \to \mathbb{R}\cup\{+\infty\}$, we consider the optimization problem
\[
\tag{$\mathcal{P}$} \min_{x \in X}\, \frac{1}{2} \nor{F(x)-y}^2 + J(x):=\Phi(x).
\]
Problem $(\mathcal{P})$ is in general a nonconvex and nonsmooth problem, having on the other hand a particular structure: it is in fact the sum of a nonconvex, smooth term and a convex and possibly  nonsmooth one. The aim of the paper is to find a convergent algorithm towards a local minimizer of $\Phi$, assuming that it exists.
Motivated by several applications  \cite{EngHanNeu96,SchGraGro09,FloPar90}, problem $(\mathcal{P})$ is receiving an increasing attention. 
In particular, for $J=0$,  $(\mathcal{P})$ is a  classical nonlinear least squares problem \cite{DenSch96,Xu09}.  This kind of problems can be solved by general optimization methods, but typically is solved by more efficient ad hoc methods. In many cases they achieve better than linear convergence, sometimes even quadratic, even though they do not need computation of second derivatives. Among the various approaches, one of the most popular is the {\em Gauss-Newton method}, introduced in \cite{Ben65}:
\beeq{eq:GNM}
{x_{n+1}=x_n- [F'(x_n)^* F'(x_n)]^{-1} F'(x_n)^* (F(x_n)-y).}
Under suitable assumptions, such a procedure is  convergent to a stationary point of  $x\mapsto\frac{1}{2} \nor{F(x)-y}^2$, namely to a point $\bar{x}$ such that $F'(\bar{x})^*(F(\bar{x})-y)=0$. Moreover,  one can easily show that the point $x_{n+1}$ defined in \eqref{eq:GNM} is the minimizer of the ``linearized'' functional:
\beeq{eq:lin}
{
x\mapsto \frac{1}{2}\nor{F(x_n)+F'(x_n)(x-x_n)-y }^2.
}
There is a wide literature devoted to the study of convergence results for the Gauss-Newton method under different perspectives. In particular we can distinguish two main streams of research: the papers devoted to a local analysis, and the ones devoted to semilocal results. The first class of studies \cite{DenSch96,ArgHil09,ArgHil10,LiZhaJin04} assume the existence of a local minimizer, and they actually determine a region of attraction around that point, meaning that if the starting point is chosen inside that region the iterative process is guaranteed to converge towards the minimizer. On the contrary, the semilocal results --- also known as Kantorovich type theorems --- do not assume the existence of a local minimizer, they just establish sufficient conditions on the starting point in order to make the   iterative procedure  convergent towards a point that is proved to be a local minimizer \cite{Arg05,FerSva09,Hau86,LiHuWan10}. 

In this paper we propose a generalization of the Gauss-Newton algorithm to the case in which $J\not=0$, that reads as follows:
\beeq{eq:prox_GN}{
x_{n+1}=\proxj\left(x_n- [F'(x_n)^* F'(x_n)]^{-1} F'(x_n)^* (F(x_n)-y)\right),}
where $\proxj$ is the proximity operator associated to $J$ (see \cite{Mor62,Mor63,Mor65}), with respect to the metric defined by the  operator  $H(x_n):=F'(x_n)^* F'(x_n)$. 
The algorithmic framework in \eqref{eq:prox_GN} is determined following the same line of \eqref{eq:lin}, i.e. linearizing the  functional $F$ at the point $x_n$, and computing the minimizer of the corresponding ``linearized'' functional
\[
x\mapsto \frac{1}{2}\nor{F(x_n)+F'(x_n)(x-x_n)-y }^2 + J(x)
\]
 This approach is the common way to deal with generalizations of the Gauss-Newton method, as we better explain in the next section.
The convergence results we obtain are of local type, and they are comparable to  those obtained for the  classical Gauss-Newton method. In particular, we get linear convergence in the general case, and quadratic convergence for zero residual problems. Furthermore, we are able to give an estimate of the radius of the convergence ball around a local minimizer.
It should be noted that the computation of the proximity operator is in general not straightforward and it may require an iterative algorithm itself, since in general a closed form is not available. 
On the other hand, we could have denoted $x_{n+1}$ simply as the minimizer of the generalized version of \eqref{eq:lin}. The formulation in terms of proximity operators allows to use the well developed theory on this kind of  operators, and in our opinion enlightens the connections and the differences with other first order methods that have recently been proposed to solve problem $(\mathcal{P})$ (see the next section for further details).  

The paper is organized as follows: we start with an analysis of the state-of-the art literature on related problems in Section \ref{sec:comp}, and then in Section \ref{sec:prel} we review the basic concepts that will be used: generalized Lipschitz conditions, generalized inverses and proximity operators. In Section \ref{sec:sett} the minimization problem is precisely stated and some necessary conditions satisfied by local minimizers are presented.  The main result of the paper, Theorem \ref{thm:1}, is discussed in Section \ref{sec:algo} and proved in Section \ref{sec:proof}. Section \ref{sec:app} gives an application to the problem of constrained nonlinear equations and, finally, in Section \ref{sec:num}  numerical tests are   set up and analyzed.

\section{Comparison with related work}\label{sec:comp}

We review the available algorithms to solve this problem, enlightening the connections and the differences with our approach.

\paragraph{Convex composite optimization} Problem $(\mathcal{P})$ can be cast, in principle, as a composite optimization problem of the form
\beeq{eq:conv_comp}
{
\min_{x\in X}\, h(c(x)),
}
by setting $h: Y\times X\to \mathbb{R}\cup\{+\infty\}$ and $c:X\rightarrow Y\times X$ as follows: 
\beeq{eq:h_and_c}
{h(s,v)=\nor{s}^2+J(v),\qquad c(x)=(F(x)-y,x).} 
Problem \eqref{eq:conv_comp} has been deeply studied in \cite{BurFer95,LewWri08,LiNg07,LiWan02,Wom85,WomFle86} from different point of views, mainly in the case $X=\mathbb{R}^n$ and $Y=\mathbb{R}^m$, but the hypotheses significantly differ, as well as the obtained results.  More specifically, in \cite{LewWri08},  the assumptions are too general to capture the features of problem $(\mathcal{P})$, and  allow only to get convergence results much weaker than the ones obtained in Theorem \ref{thm:1}.
Regarding all the remaining papers, as a matter of fact,  the following special case of {\em inclusion problem} is treated
\beeq{eq:conv_incl}
{
c(x) \in C, \qquad C = \argmin h.
}
In particular, the existence of an $x$ such that $c(x) \in  \argmin h$ is always assumed. That hypothesis is of course reasonable if we think of $h$ as a kind of norm, but if we take it as in \eqref{eq:h_and_c}, then $C=\{(s,v)\,:\, s=0, v\in\argmin J\}$ and we are lead to the condition
\[\exists\,x\in X,\quad  F(x)=y \text{ and } x \in \argmin J\]
which is too demanding for our original problem $(\mathcal{P})$.

\paragraph{Nonlinear inverse problems with regularization} Here the problem is to solve the nonlinear equation
\beeq{eq:eq}{F(x)=y} in the ill-posed case. Typically a solution is found by introducing  a regularization term weighted with a positive parameter. There are two possible approaches. The first one employs {\em iterative methods} which deal directly with problem \eqref{eq:eq}, see \cite{BacBur09}. In this case an iterative process is set up by minimizing at each step a simplified regularized problem (generally linear) having the structure of $(\mathcal{P})$ --- with a  weight for $J$ varying at each iteration. Within this class of methods, one popular choice is the {\em iteratively regularized Gauss-Newton method}, see \cite{BlaNeuSch97,Jin08,KalHof10,Lan10}.  Anyway, we remark that, despite the name, that algorithm is different from any kind of Gauss-Newton optimization method above, since it is not designed to ``optimize'' any objective functional $\Phi$, but, in fact, it directly looks for an exact solution of the equation \eqref{eq:eq}.

An alternative approach is the classical {\em Tikhonov method} \cite{EngHanNeu96} replacing \eqref{eq:eq} by the minimization of the associated Tikhonov functional. A problem  of the same type of $(\mathcal{P})$ arises,  with  $J$ weighted by a properly chosen parameter. Up to our knowledge, besides our method, the papers by Ramlau and Teschke, e.g. \cite{RamTes05,RamTes06}, are the only ones providing algorithms for the minimization of such type of functionals. In addition, the scheme they propose is different from ours and  in general it converges only weakly. On the other hand, our algorithm assumes the derivative $F^\prime(x)$ to be injective with closed range --- a hypothesis not suitable for handling the ill-posed case.

\section{Preliminaries and notations}\label{sec:prel}

We start with some notations. In the whole paper $X$ and $Y$ denote a Hilbert  space and $\Omega$ is a nonempty open subset of $X$. $F:\Omega\subseteq X\to Y$ is an operator, in general nonlinear, and $F'$ denotes its derivative (Fr\'echet or G\^ateaux). $\mathbb{L}(X,Y)$ is the space of bounded linear operators from $X$  to $Y$. If $A\in\mathbb{L}(X,Y)$, $R(A)$ and $N(A)$ shall denote respectively its range and kernel and $A^*$ its adjoint.
\subsection{Generalized Lipschitz conditions}\label{sec:lip}
Convergence results for the Gauss-Newton algorithm are typically obtained requiring Lipschitz continuity of the operator $F'$ \cite{DenSch96}. In \cite{Wan99,Wan00}, Wang introduced some weaker notions of Lipschitz continuity in the context of Newton's method. Recently, also convergence of the Gauss-Newton method has been proved under such generalized  Lipschitz conditions, see e.g. \cite{LiZhaJin04,ArgHil09,ArgHil10}. In this section we recall the definitions and review some basic properties of generalized Lipschitz continuous functions. 

First of all recall that a set $U \subset X$ is called \emph{star shaped} with respect to some of its points $x_* \in U$ if the segment $[x_*, x]$ is contained in $U$ for every $x \in U$.

\begin{definition}\label{def:radlip} Let  $f:\Omega\subseteq X\to Y$ and $U\subseteq\Omega$ be a starshaped set with respect to  $x_*\in U$. Fix $R\in(0,+\infty]$ such that $\sup_{x\in U}\nor{x-x_*}\leq R$ and let  $L:[0,R)\to\mathbb{R}$ be a positive and continuous function. The mapping $f$  is said to satisfy the {\em radius Lipschitz condition} of center $x_*$ with $L$ average on $U$ if 
 \beeq{eq:radlip}{\nor{f(x)-f(x_*+t(x-x_*))}\leq  \int_{t\nor{x-x_*}}^{\nor{x-x_*}} L(u) \ud u}
 for all $t\in[0,1]$ and $x\in U$.
 \end{definition}
 
Note that denoting by $\Gamma:[0,R)\to \mathbb{R}$ a primitive function of $L$, e.g. $\Gamma(u)=\int_0^uL(v)\ud v$, inequality \eqref{eq:radlip} can be written as
\[
\nor{f(x)-f(x_*+t(x-x_*))}\leq\Gamma(\nor{x-x_*})-\Gamma(t\nor{x-x_*}).\]
By definition $\Gamma$ is absolutely continuous and  differentiable, with $\Gamma'(u)=L(u)$. Since $L\geq 0$, $\Gamma$ is monotone increasing. Assuming $L$ to be increasing, we get that $\Gamma$ is convex. 

\begin{definition}\label{def:cenlip} Assume the hypotheses of Definition \ref{def:radlip}. We say that $f:\Omega\subseteq X\to Y$ satisfies the {\em center Lipschitz} condition  of center $x_*$ with $L$ average on $U$  if it verifies
\beeq{eq:cenlip}{
\nor{f(x)-f(x_*)}\leq\int_0^{\nor{x-x_*}} L(u) \ud u}
for every $x\in U$.
\end{definition}

The original definitions of Wang \cite{Wan99} do not require the continuity of the function $L$ but just an integrability assumption.  Our choice simplifies the proofs, but we remark that this requirement is not essential. Indeed, the proofs can be modified to handle also the more general case of Wang, at the cost of slight technical complications. In addition, the most well-known examples of Lipschitz averages are continuous. 

\begin{remark}
 If $f$ is Lipschitz continuous on a convex set $U$ with constant $L$ then it satisfies the radius and center  Lipschitz condition  of center $x_*$ with average constantly equal to $L$ on $U$, for every $x_*\in U$. Vice versa, in the definitions above, if we take $L$ constant  we obtain two intermediate concepts of Lipschitz continuity, called radius and center Lipschitz continuity, which are still in general weaker than the classical notion, being centered at a specific point. 
 
Note also that the center Lipschitz condition is weaker than the corresponding radius one. 
\end{remark}

Let $F:\Omega\subseteq X\to Y$ be a a Fr\'echet differentiable operator.  It is well-known, see e.g.  \cite{Pol83}, that if $F'$ is Lipschitz continuous with constant $L$ then the following inequality holds for every $x,x_*\in X$:
\beeq{eq:par}{
\nor{F(x_*)-F(x)-F'(x)(x_*-x)}\leq \frac L 2\nor{x_*-x}^2.}
We are going to show that under the weaker Lipschitz conditions introduced above it is still possible to prove two  estimates which are similar to \eqref{eq:par}. 
To this aim we need the following three propositions. In the first one we prove two key inequalities, and in the subsequent ones we rewrite them in a form more similar to \eqref{eq:par}.
The subsequent result is contained  implicitly in several recent papers providing local results about Gauss-Newton method, see e.g. \cite{LiWan02,LiZhaJin04}. 

\begin{proposition} Let  $F:\Omega\to Y$ be a G\^ateaux differentiable operator. Then:
\begin{itemize}
\item[(i)] if $F'$ satisfies the radius Lipschitz condition  of center $x_*$ with $L$ average on $U\subseteq\Omega$  (with $L$ and $U$ as in Definition \ref{def:radlip}), then for all $x\in U$ 
\beeq{eq:genparr}
{\nor{F(x_*)-F(x)-F'(x)(x_*-x)}\leq\int_0^{\nor{x-x_*}} L(u)u \ud u.}
\item[(ii)] if $F'$ satisfies the center Lipschitz condition   with $L$ average at $x_*$ on $U\subseteq\Omega$  (with $L$ and $U$ as in Definition \ref{def:cenlip}), then for all $x\in U$ 
\beeq{eq:genparc}
{\nor{F(x_*)-F(x)-F'(x)(x_*-x)}\leq\int_0^{\nor{x-x_*}} (2\nor{x-x_*}-u ) L(u) \ud u.}
\end{itemize}
\end{proposition}
 
\begin{proof}
Let $x \in U$ and define $\phi: [0,1] \to Y$ by setting $\phi(t) = F(x_* + t(x - x_*))$. Clearly $\phi$ is differentiable and $\phi^\prime(t) = F^\prime(x_* + t(x - x_*))(x - x_*)$. Then
\begin{equation*}
F(x) - F(x_*) = \phi(1)- \phi(0) = \int_0^1 F^\prime(x_* +t(x - x_*))(x - x_*) \ud t.
\end{equation*}
Therefore
\begin{equation*}
F(x_*) - F(x) - F^\prime(x)(x_* - x) = \int_0^1 \left(F^\prime(x_* +t(x - x_*)) - F^\prime(x)\right)(x_* - x) \ud t
\end{equation*}
and hence
\[
\nor{F(x_*) \hspace{-1pt}- \hspace{-1pt}F(x) \hspace{-1pt}-\hspace{-1pt} F^\prime(x)(x_*\hspace{-1pt} - x)} \leq \hspace{-3pt}\int_0^1\hspace{-2pt}\nor{F^\prime(x_* +t(x - x_*))\hspace{-1pt} -\hspace{-1pt} F^\prime(x)} \nor{x - x_*} \ud t.\]
Let us first prove $(i)$. 
By \eqref{eq:radlip}, we get
 \beeq{eq:dislip}{
\int_0^1\nor{F^\prime(x_* +t(x - x_*)) - F^\prime(x)} \nor{x - x_*} \ud t\leq\hspace{-1pt} \nor{x - x_*}\hspace{-1pt}\int_0^1\hspace{-4pt}\int_{t \nor{x - x_*}}^{\nor{x - x_*}}\hspace{-2pt} L(u) \ud u  \ud t.}
Note that in general, setting $\Gamma(u) = \int_0^u L(v) \ud v$, it follows
\begin{align}
\nonumber\int_0^1 \left( \int_{t \rho}^\rho L(u) \ud u \right)\rho \ud t &= \int_0^1\left( \int_0^\rho L(u) \ud u - \int_0^{t \rho} L(u) \ud u\right)\rho \ud t \\[1ex]
\nonumber&= \rho\, \Gamma(\rho) - \int_0^1 \Gamma(t \rho) \rho \ud t \\[1ex]
\nonumber& = \left[ u\, \Gamma(u)\right]_0^\rho - \int_0^\rho\Gamma(u) \ud u \\[1ex]
\nonumber& = \int_0^\rho u\, \Gamma^\prime(u) \ud u \\[1ex]
\label{eq:int}& = \int_0^\rho L(u) u \ud u
\end{align}
where we used the change of variables $u=t\rho$ and an integration by parts. Writing the equality obtained  in \eqref{eq:int} for $\rho = \nor{x - x_*}$,   \eqref{eq:dislip} becomes
\begin{equation*}
\nor{F(x_*) - F(x) - F^\prime(x)(x_* - x)} \leq \int_0^{\nor{x - x_*}}L(u) u \ud u,
\end{equation*}
so that $(i)$ is proved. \\
To show that $(ii)$ holds, observe that the center Lipschitz condition \eqref{eq:cenlip} implies
\begin{align*}
\nor{F^\prime(x_* +t(x - x_*)) - F^\prime(x)} &\leq \nor{F^\prime(x_* + t(x - x_*)) - F^\prime(x_*)} 
+ \nor{F^\prime(x_*) - F^\prime(x)} \\[1ex]
& \leq \int_0^{t\nor{x - x_*}} L(u) \ud u + \int_0^{\nor{x - x_*}}L(u) \ud u
\end{align*}
Reasoning as in the previous case and using the same notations it follows
\begin{align*}
\nor{F(x_*) - F(x) - F^\prime(x)[x_* - x]} &\leq \int_0^1 \Gamma(t \rho) \rho \ud t + \Gamma(\rho) \rho\\[1ex]
& = \int_0^\rho \Gamma(u) \ud u + \Gamma(\rho) \rho \\[1ex]
& = \int_0^{\rho} \Gamma(u) \ud u + \big[ (2 \rho - u)\Gamma(u)\big]_0^{\rho}\\[1ex]
& = \int_0^\rho (2 \rho - u) \Gamma^\prime(u) \ud u \\[1ex]
& = \int_0^{\nor{x - x_*}} (2 \nor{x - x_*} - u) L(u)\ud u.
\end{align*}

\end{proof}
The following Propositions are a direct consequence of Lemma 2.2 in \cite{LiWan03}, and have already been stated in a slightly different form in \cite{LiZhaJin04}. We include the proofs for the sake of completeness.

\begin{proposition}\label{prop:gam} Given $L:[0,R)\to\mathbb{R}$  a continuous, positive and increasing function, the function $\gamma_\lambda$ defined by setting 
\beeq{eq:gamma}
{\gamma_\lambda: [0,R)\to\mathbb{R},\qquad \gamma_\lambda(r)=\begin{cases} \dfrac{1}{r^{1+\lambda}}\displaystyle\int_0^r u^{\lambda} L(u)\ud u& \text{if }r\in(0,R)\\  &\\\dfrac{L(0)}{1+\lambda} & \text{if }r=0,  \end{cases}}
is well-defined, continuous, positive and increasing for all $\lambda\geq 0$. Moreover $\gamma_\lambda$ is constant and equal to $L/(1+\lambda)$ if $L$ is constant, 
and it is strictly increasing if $L$ is strictly increasing. Finally the following inequality holds for every $r\in [0,R)$
\beeq{eq:disgam}{
(1+\lambda)\gamma_\lambda(r)\leq L(r).}
\end{proposition}

\begin{proof} 
Clearly $\gamma_\lambda$ is differentiable on $(0,R)$ since it is a product of differentiable functions and by definition
\begin{equation}\label{eq:prod}
r^{1 + \lambda} \gamma_\lambda(r) = \int_0^r u^\lambda L(u)\ud u.
\end{equation}
Differentiating both members of \eqref{eq:prod}, it follows
\begin{equation*}
(1 + \lambda)r^\lambda \gamma_\lambda(r) + r^{1 + \lambda} \gamma_\lambda^\prime(r) = r^\lambda L(r),
\end{equation*}
therefore
\begin{equation}\label{eq:dergam}
r \gamma_\lambda^\prime(r) = L(r) - (1 + \lambda) \gamma_\lambda(r),
\end{equation}
Thus, if we prove \eqref{eq:disgam} we also get that $\gamma_\lambda(r)$ is increasing.
To this aim, taking into account that $L$ is increasing, we have
\begin{equation}\label{eq:intL}
r^{1 + \lambda} \gamma_\lambda(r) = \int_0^r u^\lambda L(u)\ud u \leq \int_0^r u^\lambda L(r)\ud u = \frac{r^{1 + \lambda}}{1 + \lambda}L(r)
\end{equation}
from which \eqref{eq:disgam} follows. Note that if $L$ is strictly increasing the inequality in \eqref{eq:intL} is strict, therefore in this case, recalling \eqref{eq:dergam},   $\gamma_\lambda^\prime(r)>0$ on $(0,R)$. On the other hand, if $L$ is constant the inequality in \eqref{eq:intL} is indeed an equality and $\gamma_\lambda^\prime(r)=0$ by \eqref{eq:dergam} implying that $\gamma_\lambda$ is constant on $(0,R)$.
The continuity of $\gamma_\lambda$ at 0 follows by L'Hospital's rule. In fact, using that $L$ is continuous at 0:
\begin{equation*}
\lim_{r \to 0} \gamma_\lambda(r) = \lim_{r \to 0} \frac{\int_0^r  u^\lambda L(u) \ud u}{r^{1 + \lambda}} = \lim_{r \to 0} \frac{r^\lambda L(r)}{(1 + \lambda) r^\lambda} = \frac{L(0)}{1 + \lambda}.
\end{equation*}

\end{proof}
Using the function $\gamma_0$ introduced in  Proposition \ref{prop:gam} the center Lipschitz condition with $L$ average can be written in the following form, resembling the classical definition of Lipschitz continuity
\[
{\nor{f(x)-f(x_*)}\leq \gamma_0(\nor{x-x_*})\nor{x-x_*}.}
\]

\begin{proposition} Under the assumptions of Proposition \ref{prop:gam}, the function 
\beeq{eq:gamc}{ \gamma^c:[0,R)\to\mathbb{R}, \qquad \gamma^c(r)=\begin{cases} \dfrac{1}{r^{2}}\displaystyle\int_0^r (2r-u) L(u) \ud u& \text{if }r\in(0,R)\\&\\ \dfrac{3L(0)}{2} & \text{if }r=0,\end{cases}  }
is well-defined, continuous, positive and increasing.
\end{proposition}

\begin{proof}
The definition of $\gamma^c$ immediately implies
 \begin{equation}\label{eq:prg}
r^2 \gamma^c(r) = \int_0^r (2 r - u)L(u) \ud u =  2 r \int_0^r L(u) \ud u - \int_0^r u L(u) \ud u
\end{equation}
and differentiating both members of \eqref{eq:prg}
\begin{equation*}
2 r \gamma^c(r) + r^2 (\gamma^c)^\prime(r) = 2 \int_0^r L(u) \ud u +  r L(r)
\end{equation*}
Dividing by $r$, and using the notations of Proposition \ref{prop:gam}, we obtain
\begin{align*}
r (\gamma^c)^\prime(r) = 2( \gamma_0(r) -  \gamma^c(r)) + L(r).
\end{align*}
Therefore, in order to prove that $\gamma^c$ is increasing we just need to show that
\beeq{eq:newdis}
{2\gamma^c(r)\leq 2\gamma_0(r)+L(r).}
In fact, using the definitions  of the functions $\gamma^c$ and $\gamma_0$ and the monotonicity of $L$  we have:
\begin{align*}
2r^2\gamma^c(r)&=2\int_0^r r L(u)\ud u+2\int_0^r (r-u)L(u)\ud u\\
& \leq 2r^2\gamma_0(r) +2 L(r)\int_0^r (r-u)\ud u\\
&=2r^2\gamma_0(r)+r^2L(r),
\end{align*}
that clearly implies \eqref{eq:newdis}. The continuity of $\gamma^c$ at $0$ can be deduced as follows
\[
\lim_{r\to 0}\gamma^c(r)=\lim_{r\to 0} 2\gamma_0(r)-\gamma_1(r)=2L(0)-\frac{L(0)}{2}=\frac 3 2 L(0).
\]  
relying on the continuity of $\gamma_0$ and $\gamma_1$ proved in Proposition \ref{prop:gam}. 

\end{proof}

\begin{remark}\label{rem:gam}
Using the functions $\gamma_0, \gamma_1$ and $\gamma^c$ introduced in \eqref{eq:gamma} and in \eqref{eq:gamc}, the 
inequality  \eqref{eq:cenlip} written for $F'$ becomes 
\begin{align*}
\nor{F'(x)-F'(x_*)}\leq \gamma_0(\nor{x-x_*})\nor{x-x_*}.
\end{align*}
The inequalities \eqref{eq:genparr} and \eqref{eq:genparc} can be written respectively as
\begin{align}
\nor{F(x_*) - F(x) - F^\prime(x)(x_* - x)} &\leq  \gamma_1(\nor{x - x_*})\nor{x - x_*}^2
\end{align}
and
\begin{align}
\nor{F(x_*) - F(x) - F^\prime(x)(x_* - x)} &\leq  \gamma^c(\nor{x - x_*}) \nor{x - x_*}^2,
\end{align}
that generalize the inequality \eqref{eq:par}.

Note moreover that the functions $r\gamma_0(r),r^2\gamma_1(r)$ and $r^2\gamma^c(r)$ are always strictly increasing and if $L$ is a constant function equal to $L$, then 
\[
\gamma_0(r)=L,\qquad \gamma_1(r)=\frac L 2,\qquad \gamma^c(r)=\frac 3 2 L.
\]
\end{remark}
\begin{remark}
 It is worth noting that, even though only a G\^ateaux differentiability has been required, the previous inequalities together with the hypothesis on the function $L$ implies Fr\'echet differentiability at $x_*$. 
\end{remark}

\subsection{Generalized inverses} 

In this section we collect some well-known results regarding the {\em Moore-Penrose generalized inverse} (also known as  {\em pseudoinverse}) $A^\dag$ of a linear operator $A$. They will be useful in the rest of the paper. For the definition and a comprehensive analysis of the properties of the Moore-Penrose inverse we refer the reader to \cite{Gro77}. 

Assume that  $A\in \mathbb{L}(X,Y)$ has a  closed range. The pseudoinverse of $A$  is the linear operator $A^\dag\in\mathbb{L}(Y,X)$ defined by means of the four ``Moore-Penrose equations''

\begin{equation}
\label{lem:gids1} AA^\dag A=A, \quad A^\dag AA^\dag =A^\dag, \quad (AA^\dag)^*=AA^\dag,\quad (A^\dag A)^*=A^\dag A.\\
\end{equation}
Denoting by $P_{N(A)}$ and $P_{{R(A)}}$ the orthogonal projectors onto the kernel and the  range of $A$ respectively, from the definition it is clear that
\begin{equation}
\label{lem:gids2}A^\dag A=I-P_{N(A)}, \qquad AA^\dag=P_{{R(A)}}.
\end{equation}  
In case $A$ is injective, $N(A)=\{0\}$ and $A^\dag A=I$, that is $A^\dag$ is a left inverse of $A$.
Furthermore, for each $A \in\mathbb{L}(X,Y)$ the following statements are equivalent:
\begin{itemize}
\item $A$ is injective and the range of $A$ is closed;
\item  $A^*A$ is invertible in $\mathbb{L}(X,X)$. 
\end{itemize}
and if one of those equivalent conditions is true then $A^\dag=(A^*A)^{-1}A^*$ and $\nor{A^\dag}^2=\nor{(A^*A)^{-1}}$.

The following lemma gives a perturbation bound for the Moore-Penrose pseudoinverse, see \cite{Ste69,Wed73}.
\begin{lemma}\label{lem:daga} Let $A,B \in \mathbb{L}(X,Y)$ with $A$ injective and $R(A)$ closed. If $\nor{(B-A)A^\dag}<1$, then $B$ is injective, $R(B)$ is closed and 

\begin{equation*}
\|B^\dag\| \leq \frac{\|A^\dag\|}{1 - \nor{(B-A)A^\dag}}.
\end{equation*}
Moreover
\begin{equation*}
\|B^\dag - A^\dag\| \leq \sqrt{2} \|A^\dag\| \|B^\dag\| \nor{B - A},
\end{equation*}
and therefore
\begin{equation*}
\|B^\dag - A^\dag\| \leq \sqrt{2} \frac{\|A^\dag\|^2 \nor{B-A}}{1 - \nor{A^\dag} \nor{B-A}}.
\end{equation*}
\end{lemma}

\subsection{The proximity operator }\label{sec:prox} 

This section consists of an introduction on proximity operators, which were first introduced by Moreau in \cite{Mor62}, and further investigated in \cite{Mor63,Mor65} as a generalization of the notion of convex projection operator. 
Let $H:X\to X$ be a continuous, positive and selfadjoint operator, bounded from below, and therefore invertible. Then we can define a new scalar product on $X$ by setting $\langle x,z\rangle_H=\langle x, Hz\rangle $. The corresponding induced norm $\nor{\cdot}_H$ is equivalent to the given norm on $X$, since the following inequalities hold true
\beeq{eq:norm_eq}
{\frac{1}{\nor{H^{-1}}}\nor{x}^2\leq\nor{x}_H^2\leq \nor{H}\nor{x}^2.}

The \emph{Moreau-Yosida approximation} of a convex and lower semicontinuous function $\varphi:X\to\rone\cup\{+\infty\} $ with respect to the scalar product induced by $H$ is the function $M_{\varphi}:X\to\rone$ defined by setting
\beeq{eq:MY}
{M_{\varphi}(z)=\inf_{x\in X}\left\{\varphi(x)+\frac{1}{2}\nor{x-z}_H^2 \right\}.}

For every $z\in X$, the infimum in equation \eqref{eq:MY} is attained at a unique point, denoted $\prox^{H}_\varphi(z)$. In this way, an operator 
\begin{equation*}
\prox^{H}_\varphi : X \to X
\end{equation*}
is defined, which is called the \emph{proximity operator} associated to $\varphi$ w.r.t.  $H$. In case $H=I$ is the identity, the proximity operator is denoted simply by $\prox_\varphi$. Writing the first order optimality conditions for \eqref{eq:MY}, we get
\beeq{eq:prox_first}
{p=\prox_\varphi^{H}(z) \Longleftrightarrow  0\in \partial \varphi(p)+H(p-z) \Longleftrightarrow Hz\in (\partial \varphi+H)(p),}
which gives
\[
\prox_{\varphi}^{H}(z) = (H+\partial \varphi)^{-1}(Hz).\]

We remark that the  map $(H+\partial \varphi)^{-1}$ (in principle multi-valued) is single-valued, since we know that the minimum is attained at a unique point.

\begin{lemma} \label{lem:prox_lip} The proximity operator $\prox_\varphi^{H}:X\to X$ is Lipschitz with  constant $\sqrt{\nor{H}\nor{H^{-1}}}$ with respect to $\nor{\cdot}$, namely
\beeq{eq:non_exp}
{\nor{\prox_\varphi^{H}(z_1)-\prox_\varphi^{H}(z_2)}\leq \sqrt{\nor{H}\nor{H^{-1}}}\nor{z_1-z_2}.}
 \end{lemma}
\begin{proof}
Being the proximity operator firmly nonexpansive with respect to the scalar product induced by $H$ (see e.g. Lemma 2.4 in \cite{ComWaj05}) we have 
\[
\nor{\prox_\varphi^{H}(z_1)-\prox_\varphi^{H}(z_2)}_H\leq\nor{z_1-z_2}_H.
\]
Using the inequalities in \eqref{eq:norm_eq} relating $\nor{\cdot}$ and $\nor{\cdot}_H$ we get the desired result. 
\end{proof}

It is also possible to show that in some cases the computation of the proximity operator with respect to the scalar product induced by $H$ can be brought back to the computation of the proximity operator with respect to the original norm. In particular, the following proposition holds.

\begin{proposition}
Let $A \in\mathbb{L}(X,Y)$. Let us suppose $A$ to be injective with closed range. Set $H = A^*A$ and assume $\varphi: X \to \mathbb{R}\cup\{+\infty\}$ a proper, convex and lower semicontinuous functional. Then
\begin{equation*}
\mathrm{prox}_\varphi^{H} = A^\dag \mathrm{prox}_{\varphi\circ A^\dag} A
\end{equation*}
\end{proposition}

\begin{proof}
Being $A$ injective, $H$ is positive and invertible, thus
\begin{equation*}
\mathrm{prox}_\varphi^H : X \to X \qquad \mathrm{prox}_\varphi^H(x) = \argmin_{z \in X} \Big\{ \varphi(z) + \frac{1}{2}\nor{z - x}_H^2 \Big\}
\end{equation*}
where
\begin{equation*}
\nor{z - x}_H^2 = \langle A^* A (z - x), z- x\rangle = \langle A(z - x), A(z - x)\rangle = \nor{A(z - x)}^2.
\end{equation*}
Therefore
\begin{equation}\label{eq:proA}
\mathrm{prox}_\varphi^H(x) = \argmin_{z \in X} \Big\{ \varphi(z) + \frac{1}{2}\nor{A(z - x)}^2 \Big\}.
\end{equation}
On the other hand, since $\varphi\circ A^\dag : Y \to \rone\cup\{+\infty\}$ is convex and lower semicontinuous the corresponding proximity operator with respect to $\nor{\cdot}$ is well-defined and
\begin{equation*}
\mathrm{prox}_{\varphi\circ A^\dag} : Y \to Y \qquad \mathrm{prox}_{\varphi\circ A^\dag}(y) = \argmin_{t \in Y} \Big\{ \varphi(A^\dag t) + \frac{1}{2} \nor{t - y}^2 \Big\}
\end{equation*}
If we set $\overline{x} = \mathrm{prox}_\varphi^H(x)$, by \eqref{eq:proA} we obtain
\begin{equation*}
\varphi(\overline{x}) + \frac{1}{2} \nor{A \overline{x} - A x}^2 \leq \varphi(z) + \frac{1}{2}\nor{A z - A x}^2 \quad \forall z \in X.
\end{equation*}
Moreover, by setting $\overline{y} = A \overline{x}$, taking $t \in Y$, with $t = t_1 + t_2$ such that $t_1 \in R(A)$ and $t_2 \in R(A)^\perp$ and $z = A^\dag t \in X$, we have
\begin{align*}
\varphi(A^\dag t) + \frac{1}{2}\nor{t - A x}^2 &= \varphi(A^\dag t) + \frac{1}{2}\nor{t_1 + t_2 - A x}^2 \\
& = \varphi( A^\dag t) + \frac{1}{2}\nor{P_{R(A)} t - A x}^2 + \nor{t_2}^2 \\
& \geq \varphi(A^\dag t) + \frac{1}{2}\nor{A A^\dag t - A x}^2 \\
& = \varphi(z) + \frac{1}{2}\nor{A z - A x}^2 \\
& \geq \varphi(\overline{x}) + \frac{1}{2}\nor{A \overline{x} - A x}^2 \\
& = \varphi(A^\dag \overline{y}) + \frac{1}{2}\nor{\overline{y} - A x}^2. \\
\end{align*}
We finally get
\begin{equation*}
\overline{y} = \argmin_{t \in Y} \Big\{ \varphi (A^\dag t) + \frac{1}{2}\nor{ t - A x}^2 \Big\} = \mathrm{prox}_{\varphi\circ A^\dag}(A x)
\end{equation*}
and  thus using \eqref{lem:gids2}
\begin{equation*}
\mathrm{prox}_\varphi^H(x) = \overline{x} = A^\dag A \overline{x} = A^\dag \overline{y} = A^\dag \mathrm{prox}_{\varphi\circ A^\dag}(A x).
\end{equation*}

\end{proof}

Since in the sequel the proximity operators will be computed with respect to a variable norm $\nor{\cdot}_H$, we are interested in the behavior of the proximity operator when $H$ varies.

\begin{lemma}\label{lem:H_varia} Let $H_1$ and $H_2$ two continuous positive selfadjoint operators on $X$, both bounded  from below. It holds
\beeq{eq:H_varia}
{\nor{\prox_\varphi^{H_1}(z)-\prox_\varphi^{H_2}(z)}\leq \nor{H_1^{-1}}\nor{(H_1-H_2)(z-\prox_\varphi^{H_2}(z))}.  }
\end{lemma}
\begin{proof}  
By \eqref{eq:prox_first} it follows
\[
H_i(z-\prox_\varphi^{H_i}(z))\in \partial \varphi(\prox_\varphi^{H_i}(z)),\quad i=1,2. \]
Then, by definition of subdifferential the following inequalities hold true
\begin{align*}
\varphi(\prox_\varphi^{H_2}(z)) \geq &\varphi(\prox_\varphi^{H_1}(z))\\ &\qquad+\langle H_1(z-\prox_\varphi^{H_1}(z)), \prox_\varphi^{H_2}(z)-\prox_\varphi^{H_1}(z)\rangle \end{align*}
\begin{align*}
\varphi(\prox_\varphi^{H_1}(z)) \geq &\varphi(\prox_\varphi^{H_2}(z))\\ &\qquad+\langle H_2(z-\prox_\varphi^{H_2}(z)), \prox_\varphi^{H_1}(z)-\prox_\varphi^{H_2}(z)\rangle .
 \end{align*}
Summing up them, we obtain
\[
0\geq \langle H_2(z-\prox_\varphi^{H_2}(z))-H_1(z-\prox_\varphi^{H_1}(z)), \prox_\varphi^{H_1}(z)-\prox_\varphi^{H_2}(z)\rangle,
\]
and equivalently
\begin{multline*}
\langle H_1\prox_\varphi^{H_1}(z)-H_2\prox_\varphi^{H_2}(z),\prox_\varphi^{H_1}(z)-\prox_\varphi^{H_2}(z)\rangle\leq\\ \langle (H_1-H_2)z,\prox_\varphi^{H_1}(z)-\prox_\varphi^{H_2}(z)\rangle. 
\end{multline*}
Adding and subtracting the same term, the previous inequality can also be written as
\begin{multline*}
\langle H_1(\prox_\varphi^{H_1}(z)-\prox_\varphi^{H_2}(z)),\prox_\varphi^{H_1}(z)-\prox_\varphi^{H_2}(z)\rangle\leq\\ \langle (H_1-H_2)(z-\prox_\varphi^{H_2}(z),\prox_\varphi^{H_1}(z)-\prox_\varphi^{H_2}(z)\rangle,
\end{multline*}
from which \eqref{eq:H_varia} follows. 
\end{proof}
Note that in the previous lemma  $H_1$ and $H_2$ play a symmetric role, so that they can be interchanged. 
\begin{remark}\label{rem:vartut} Combining \eqref{eq:non_exp} and \eqref{eq:H_varia}, we get:
\begin{align}
\nonumber\nor{\mathrm{prox}_J^{H_1} z_1 - \mathrm{prox}_J^{H_2} z_2} &\leq \nor{\mathrm{prox}_J^{H_1} z_1 - \mathrm{prox}_J^{H_1} z_2} + \nor{\mathrm{prox}_J^{H_1} z_2 - \mathrm{prox}_J^{H_2} z_2} \\
\label{eq:vartut}&\leq \big(\nor{H_1} \nor{H_1^{-1}}\big)^{1/2} \nor{z_1 - z_2} \\
\nonumber&\qquad\qquad+ \nor{H_1^{-1}}\nor{(H_1 - H_2)(z_2 - \mathrm{prox}_J^{H_2} z_2)},
\end{align}
for every $z_1,z_2\in X$ and $H_1,H_2$ continuous and positive selfadjoint operators on $X$, bounded  from below.
\end{remark}

\section{Setting the minimization problem}\label{sec:sett}

In this section we collect some basic properties of the solutions of problem $(\mathcal{P})$. The following will be standing hypotheses throughout the paper. 
\[(SH)\qquad
\left\{\begin{array}{l}
F:\Omega\subseteq X\to Y \text{ is G\^ateaux differentiable}\\

\ \\

J:X\to \rone\cup\{+\infty\} \text{ is proper, lower semicontinuous and convex. }
\end{array}\right.
\]
Recall that $J$ proper means the {\em effective domain} $\mathrm{dom}(J):=\{x\in X\,:\, J(x)< +\infty\}$ is nonempty.

Without loss of generality, we shall assume $y=0$ in the problem $(\mathcal{P})$, since the general case can be recovered just by replacing $F$ with $F-y$. Thus, hereafter, the following optimization problem will be considered 
\[
\tag{$\mathcal{P}_0$} \min_{x \in X}\, \frac{1}{2} \nor{F(x)}^2 + J(x):=\Phi_0(x).
\]
The functional $\Phi_0$ is in general nonconvex and searching for global minimizers turns out to be a challenging task. Therefore, the focus of this paper is on local minimizers of $\Phi_0$, whose existence shall be assumed from now on. Generally speaking,  Gauss-Newton methods are of a local character, and allow to find a local minimizer. As it is well-known a \emph{local minimizer }$x_*$ of  $\Phi_0$ is a point such that $x_*\in\mathrm{dom} (J)\cap\Omega$ and there exists a neighborhood $U$ of $x_*$ such that $\Phi_0(x_*)\leq\Phi_0(x)$ for all $x\in U$. 

By the way, hypotheses $(SH)$ are not enough to guarantee the existence of a global minimizer of the problem  $(\mathcal{P}_0)$. Such existence can be proved relying on the Weierstrass theorem as soon as we impose $F$ to be weak to weak continuous and $\Phi_0$ (weakly) coercive, namely $\lim_{\nor{x}\rightarrow+\infty} \Phi_0(x)=+\infty.$ 

We start by providing first order conditions for local minimizers.

\begin{proposition}\label{prop:ex} Suppose $(SH)$ are satisfied and let $x_*\in \Omega$ be a local minimizer of $\Phi_0$. Then the following {\em stationary condition} holds
\[-F'(x_*)^*F(x_*)\in\partial J(x_*).\]
Moreover, if $F'(x_*)$ is injective and $R(F'(x_*))$ is closed, then  $x_*$ satisfies the {\em fixed point equation}
\[x_*=\prox_J^{H(x_*)}(x_*-F'(x_*)^\dag F(x_*)),\] with $H(x_*):=F'(x_*)^*F'(x_*)$.
\end{proposition}
\begin{proof}

Suppose that $x_*$ is a local minimizer of $\Phi_0$. Denoting by $\Phi_0'(x_*,v)$ the directional derivative of $\Phi_0$ at $x_*$ in the direction $v\in X$, which exists thanks to $(SH)$, the first order optimality conditions for $x_*$ implies 
\beeq{eq:Cla}
{\Phi_0'(x_*,v)\geq0 \qquad\forall v\in X.}
As a consequence of the differentiability of $F$ and the convexity of $J$ \eqref{eq:Cla} can be rewritten as
\[
-F'(x_*)^*F(x_*)v \leq J'(x,v)\qquad\forall v\in X,
\]
and consequently, by Proposition 3.1.6 in \cite{borwein99}, also as
\beeq{eq:sub}
{-F'(x_*)^*F(x_*)\in\partial J(x_*), }
 which is the stationary condition of the thesis. 
To prove that $x_*$ satisfies the fixed point equation note that adding $H(x_*)x_*$ to both members of \eqref{eq:sub} we have
\[
H(x_*)x_*-F'(x_*)^*F(x_*)\in (H(x_*)+\partial J)(x_*). 
\]
Since $H(x_*)$ is invertible, then the previous equation can be also rewritten as
\[
H(x_*)(x_*-H(x_*)^{-1}F'(x_*)^*F(x_*))\in (H(x_*)+\partial J)(x_*). 
\]
Recalling equation \eqref{eq:prox_first} and the properties enjoyed by the pseudoinverse we obtain the second assertion.
\end{proof}

\section{The algorithm - convergence analysis}\label{sec:algo}

In this section we state the main result of the paper, consisting in the study of the convergence of a generalized Gauss-Newton method for solving problem $(\mathcal{P}_0)$. The flavor  is similar to the most recent results concerning the standard Gauss-Newton method, proved  in \cite{LiZhaJin04}. We start describing some basic properties of the proposed algorithmic framework.

Fix $x_0\in \mathrm{dom}(J)$, and then, given $x_n$, define $x_{n+1}$ by setting
\beeq{eq:xn}
{ x_{n+1}= \argmin_{x\in X} \frac 1 2\nor{F(x_n)+F'(x_n)(x-x_n)}^2+J(x).
}
Note that since the quantity inside the norm has been linearized, this problem can be solved explicitly, for instance using first order methods for the minimization of nonsmooth convex functions, such as bundle methods or forward-backward methods (see \cite{HirLem93,ComWaj05}).
Writing down the first order optimality conditions we will get a similar formula to the one in Proposition \ref{prop:ex} for a minimizer $x_*$.
 
\begin{proposition} \label{prop:fix_n}  Suppose $F'(x_n)$ is injective with closed range and set $H(x_n)=F'(x_n)^*F'(x_n)$. Then, the formula   \eqref{eq:xn} defining $x_{n+1}$  is equivalent to
\beeq{eq:xn1}{
x_{n+1}=\proxj (x_n-F'(x_n)^\dag F(x_n)).}
 \end{proposition}
 \begin{proof} Thanks to  the assumptions made on $F'(x_n)$ the operator $H(x_n)$ is invertible. Writing the first order necessary conditions, which are satisfied by $x_{n+1}$ we obtain
\begin{eqnarray*}
&&0\in F^\prime(x_n)^*[F(x_n) + F^\prime(x_n)(x_{n+1} - x_n)] + \partial J(x_{n+1}) \\[1ex]
&\iff& F^\prime(x_n)^* F^\prime(x_n) x_n - F^\prime(x_n)^* F(x_n) \in (F^\prime(x_n)^* F^\prime(x_n) + \partial J)(x_{n+1}) \\[1ex]
&\iff& x_{n+1} = (F^\prime(x_n)^* F^\prime(x_n) + \partial J)^{-1}( F^\prime(x_n)^* F^\prime(x_n) x_n - F^\prime(x_n)^* F(x_n)) \\[1ex]
&\iff& x_{n+1} = (F^\prime(x_n)^* F^\prime(x_n) + \partial J)^{-1}  F^\prime(x_n)^* F^\prime(x_n) \left( x_n - F^\prime(x_n)^\dag F(x_n)\right) \\[0.8ex]
&\iff& x_{n+1} = \prox_j^{H(x_n)}\left( x_n - F^\prime(x_n)^\dag F(x_n) \right)
\end{eqnarray*}

 \end{proof}
 
 In the next theorem we provide a local convergence analysis of the proximal Gauss-Newton method, under the generalized Lipschitz conditions on $F'$ introduced in Section \ref{sec:lip}. The proof is postponed to Section \ref{sec:proof}.

 \begin{theorem}\label{thm:1}
Suppose that (SH) are satisfied.  Let $U\subseteq\Omega$ be an open starshaped set with respect to $x_*$, where   $x_* \in \mathrm{dom} (J)\cap U$ is a  local minimizer of $\Phi_0$. Moreover assume
\begin{enumerate}
\item $F'(x_*)$ is injective with closed range;\\[0.3ex]
\item $F^\prime:\Omega\subseteq X \to \mathbb{L}(X,Y)$ is center Lipschitz continuous  of center $x_*$ with $L$ average on $U$ ($L$ as in the definition \ref{def:cenlip} and increasing); \\[0.3ex]
\item $[(1 + \sqrt{2})\kappa + 1] \alpha \beta^2 L(0)<1$, where $\alpha=\nor{F(x_*)},\beta= \nor{F'(x_*)^\dag}$, $\kappa=\nor{F'(x_*)^\dag}\nor{F'(x_*)}$, the conditioning number of $F'(x_*)$.
\end{enumerate}
\ \\[0.3ex]
Define $\bar{R}$ and $q:[0,\bar{R})\to\rone_+$ by setting
$\bar{R}=\sup\{r\in (0,R)\,:\, \gamma_0(r)r<1/\beta\}$ and 
\begin{align*}
q(r) &= \frac{ \beta}{1 - \beta \gamma_0(r) r} \bigg\{ \frac{ \beta \gamma_0(r) \gamma^c(r) r^2+ \kappa \gamma^c(r) r}{(1 - \beta \gamma_0(r) r)} + \frac{(1 + \sqrt{2}) \alpha \beta^2 \gamma_0(r)^2 r }{1 - \beta \gamma_0(r) r} \\[1.5ex]
&\qquad \qquad\qquad \qquad\qquad \qquad\qquad \qquad+ \frac{[(1 + \sqrt{2})\kappa + 1] \alpha \beta \gamma_0(r)}{1 - \beta \gamma_0(r) r}  \bigg\}.
\end{align*} 
The function $q$ is continuous and strictly increasing. If we define
\beeq{eq:rsegnato}{\bar{r}= \sup\{r\in(0,\bar{R}]\,:\, q(r)<1\},}
and we fix $r\in\rone$, with $0<r\leq \bar{r}$, such that $B_r(x_*)\subseteq U$, we get that the sequence
\begin{eqnarray*}
&&x_0 \in B_r(x_*),\\
&&x_{n+1} = \mathrm{prox}_{J}^{H(x_n)} \big(x_n - F^\prime(x_n)^\dag F(x_n)\big)
\end{eqnarray*}
with  $H(x_n):=F^\prime(x_n)^* F^\prime(x_n)$, is well-defined, i.e. $x_n \in B_r(x_*)$ and $F^\prime(x_n)$ is injective with closed range and it holds
\begin{equation*}
\nor{x_n - x_*} \leq q_0^n \nor{x_0 - x_*},
\end{equation*}
where $q_0:=q(\nor{x-x_0})<1$.
\ \\
More precisely, the following inequality is true 
\[
\nor{x_{n+1} - x_*}\leq C_2\nor{x_n-x_*}^2+C_1\nor{x_n-x_*},
\]
for  constants $C_1\geq 0$ and $C_2>0$ defined as
\begin{eqnarray*}
C_1&=&\frac{[(1 + \sqrt{2})\kappa + 1] \alpha \beta^2 \gamma_0(\rho_{x_0})}{(1 - \beta \gamma_0(\rho_{x_0}) \rho_{x_0})^2}; \\
C_2&=&\frac{\kappa\beta \gamma^c(\rho_{x_0})+ (1 + \sqrt{2}) \alpha \beta^3 \gamma_0(\rho_{x_0})^2 + \beta^2 \gamma_0(\rho_{x_0}) \gamma^c(\rho_{x_0})\rho_{x_0}}{(1 - \beta \gamma_0(\rho_{x_0}) \rho_{x_0})^2},
\end{eqnarray*}
with $\rho_{x_0}=\nor{x-x_0}$.
\end{theorem}
Since $\bar{r}$ is chosen as the biggest value ensuring $q(r)\leq 1$  (a sufficient condition making the Gauss-Newton sequence convergent), $\bar{r}$  can be thought as the radius of the basin of attraction around the local minimum point $x_*$, even though in general we can't prove the optimality of this value.

\begin{remark} An analogous theorem is true if in the assumption 2 we suppose $F'$ to satisfy the radius Lipschitz condition. All the statements remain true, just replacing $\gamma^c$ with $\gamma_1$. 
For instance the expression of $q(r)$ becomes
\begin{align*}
q(r) &= \frac{ \beta}{1 - \beta \gamma_0(r) r} \bigg\{ \frac{ \beta \gamma_0(r) \gamma_1(r) r^2+ \kappa \gamma_1(r) r}{1 - \beta \gamma_0(r) r} + \frac{(1 + \sqrt{2}) \alpha \beta^2 \gamma_0(r)^2 r }{1 - \beta \gamma_0(r) r} \\[1.5ex]
&\qquad \qquad\qquad \qquad\qquad \qquad\qquad \qquad+ \frac{[(1 + \sqrt{2})\kappa + 1] \alpha \beta \gamma_0(r)}{1 - \beta \gamma_0(r) r}  \bigg\}.
\end{align*} 
\end{remark}

\begin{remark} 
The hypotheses we impose are in line with the state-of-art literature about classical Gauss-Newton method ($J=0$), see \cite{LiZhaJin04}. It is worth noting that the expression of $\bar{r}$ is not affected by the choice of $J$.
On the other hand, the presence of the function $J$ reduces the radius of convergence of the Gauss-Newton method. Indeed, the expression for $\bar{r}$ obtained in \eqref{eq:rsegnato}, which is  valid also in the case $J=0$ is always smaller than the maximum radius of convergence  that can  be derived from  equation (3.4)  in  \cite{LiZhaJin04}, namely 
\[
r_0=\sup\{r\in(0,\bar{R})\,:\, q_0(r)<1\},
\]
with 
\[
q_0(r)=\frac{ \beta}{1 - \beta \gamma_0(r) r} \big\{ \gamma^c(r)r+\sqrt{2}\beta\alpha\gamma_0(r) \big\}.
\]
The reason is that the bound \eqref{eq:non_exp} we use, is not sharp in case $J=0$, and this causes an additional term in the expression of $q(r)$. 
\end{remark}

{\em Conditions ensuring quadratic convergence.} As in the classical case, also with the additional term $J$, for zero residual problems quadratic convergence holds. In fact, from the expression of $C_1$, we see that $C_1=0$ if $\alpha=0$, i.e. $F(x_*)=0$.
\ \\
\subsection{The case of constant average $L$}
In case the function $L$ is constant, we can derive also an explicit expression for the maximum ray of convergence $\bar{r}$. 

\begin{corollary} Let the assumptions of Theorem \ref{thm:1} be satisfied and moreover assume $F'(x_*)$ to be center Lipschitz continuous of center $x_*$ with {\em constant} average $L$ on $U$.  Define $q:[0,1/(\beta L))\to\mathbb{R}_+$ as
\beeq{eq:ql}{
q(r) = \frac{ \beta}{1 - \beta L r} \bigg\{ \frac{ 3(\beta L^2 r^2+ \kappa L r)}{2(1 - \beta L r)} + \frac{(1 + \sqrt{2}) \alpha \beta^2 L^2 r }{1 - \beta L r} + \frac{[(1 + \sqrt{2})\kappa + 1] \alpha \beta L}{1 - \beta L r}\bigg\},}
which is continuous and strictly increasing in its domain. If we define 
\begin{align*}
h&=[(1 + \sqrt{2})\kappa + 1] \alpha \beta^2 L\quad(<1),\\ 
\bar{r}&=\frac{1}{\beta L}\left[-\left(2+\frac {3\kappa} 2+(1+\sqrt{2})\alpha\beta^2L\right)+\sqrt{\left(2+\frac{3 \kappa} 2+(1+\sqrt{2})\alpha\beta^2L\right)^2+2(1-h)}\,\right]
\end{align*}
and we fix $r\in\mathbb{R}$ with $0<r\leq \bar{r}$ such that $B_r(x_*)\subseteq U$ the conclusions of Theorem \ref{thm:1} hold with 
 \begin{eqnarray*}
C_1&=&\frac{[(1 + \sqrt{2})\kappa + 1] \alpha \beta^2 L}{(1 - \beta L\rho_{x_0})^2}; \\
C_2&=&\beta\frac{3\kappa + (1 + \sqrt{2}) \alpha \beta^2 L^2 + 3\beta L^2\rho_{x_0}}{2(1 - \beta L \rho_{x_0})^2}.
\end{eqnarray*}

\end{corollary}

\begin{proof}

Using the expressions of $\gamma_0$ and $\gamma^c$ found in Remark \ref{rem:gam},  one can easily show that $\overline{R}=1/(\beta L)$ and $q$ can be written as in \eqref{eq:ql}.

As before $q$ is continuous and strictly increasing on the interval $[0,1/(\beta L))$, and
\[
q(0)=h<1,\qquad \lim_{r\to1/(\beta L)} q(r)=+\infty.
\]
Therefore $\bar{r}$ defined in \eqref{eq:rsegnato} is the unique solution in $(0,1/(\beta L))$ of the equation $q(r)=1$.   We are going to find that point explicitly by solving  the equation $q(r)=1$.
The latter  is equivalent to  the following quadratic equation
\[
z^2+(4+3\kappa+2(1+\sqrt{2})\alpha\beta^2L)z -2(1-h)=0,
\] 
with $z\in [0,1)$, $z=\beta Lr$. That equation has the two distinct solutions
\[
z=-\left(2+\frac 3 2\kappa+(1+\sqrt{2})\alpha\beta^2L\right)\pm\sqrt{\left(2+\frac 3 2\kappa+2(1+\sqrt{2})\alpha\beta^2L\right)^2+2(1-h)}.
\]
Of course, we discard the negative solution, and we keep the one with the plus sign, which can be easily checked to belong to $(0,1)$.

\end{proof}

Along the same line,  a similar result concerning the case of radius Lipschitz continuity can be proved.

\begin{corollary} Let the assumptions of Theorem \ref{thm:1} be satisfied and moreover assume $F'(x_*)$ to be radius Lipschitz continuous of center $x_*$ with {\em constant} average $L$ on $U$.  Define $q:[0,1/(\beta L))\to\mathbb{R}_+$ as
\[
q(r) = \frac{ \beta}{1 - \beta L r} \bigg\{ \frac{ \beta L^2 r^2+ \kappa L r}{2(1 - \beta L r)} + \frac{(1 + \sqrt{2}) \alpha \beta^2 L^2 r }{1 - \beta L r} + \frac{[(1 + \sqrt{2})\kappa + 1] \alpha \beta L}{1 - \beta L r}\bigg\},\]
which is continuous and strictly increasing in its domain. If we define 
\begin{align*}
h&=[(1 + \sqrt{2})\kappa + 1] \alpha \beta^2 L(<1),\\ 
\bar{r}&=\frac{1}{\beta L}\left[\left(2+\frac {\kappa} 2+(1+\sqrt{2})\alpha\beta^2L\right)-\sqrt{\left(2+\frac{ \kappa} 2+(1+\sqrt{2})\alpha\beta^2L\right)^2-2(1-h)}\,\right]
\end{align*}
and we fix $r\in\mathbb{R}$ with $0<r\leq \bar{r}$ such that $B_r(x_*)\subseteq U$ the conclusions of Theorem \ref{thm:1} hold with 
 \begin{eqnarray*}
C_1&=&\frac{[(1 + \sqrt{2})\kappa + 1] \alpha \beta^2 L}{(1 - \beta L\rho_{x_0})^2}; \\
C_2&=&\beta\frac{\kappa + (1 + \sqrt{2}) \alpha \beta^2 L^2 + \beta L^2\rho_{x_0}}{2(1 - \beta L \rho_{x_0})^2}.
\end{eqnarray*}
for $r<1/(\beta L)$.
\end{corollary}

\section{Proof of Theorem \ref{thm:1}} \label{sec:proof}

We are going to state some auxiliary results for proving convergence of the algorithm discussed in the previous section.  The following proposition will be one of the building blocks to show that convergence holds. 

\begin{proposition}\label{prop:fp} Let $G:D\subseteq X\to X$, be a mapping and $x_*\in D$ a fixed point of $G$. Let  $U\subseteq D$ be an open starshaped set with respect to $x_*\in U$.   Assume  $G$ to satisfy the inequality
\beeq{eq:fp}
{
\nor{G(x)-G(x_*)}\leq q(\nor{x-x_*})\nor{x-x_*}, \qquad \text{for all  } x\in U
}
for a given  increasing function $q:[0,\overline{R})\to [0,+\infty)$, continuous at 0 and such that $q(0)<1$. Define 
$$\bar{r}=\sup\{r\in(0,\overline{R})\,:\, q(r)< 1\}.$$ 
Then $\bar{r}>0$ and given $r\in\mathbb{R}$ with $0<r\leq\bar{r}$ and $B_r(x_*)\subseteq U$, it follows  $G(B_{r}(x_*))\subseteq B_{r}(x_*)$, thus, given $x_0\in B_{r}(x_*)$ the sequence defined by setting $x_{n+1}=G(x_n)$ is well-defined. Moreover, denoting $q_0=q(\nor{x_0-x_*})$ it holds $q_0<1$ and
$$\nor{x_{n+1}-x_*}\leq q_0^n\nor{x_0-x_*}.$$
\end{proposition}
\begin{proof}  First note that $\bar{r}=\sup\{r\in(0,R)\,:\, q(r)< 1\}>0$ being $q(0)<1$ and $q$ continuous at 0. 
Fix $r\in\mathbb{R}$, $0<r\leq\bar{R}$ such that $B_r(x_*)\subseteq U$ and $x\in B_{r}(x_*)$.  Then by definition of $\bar{r}$, $q(\nor{x-x_*})<1$ and therefore \eqref{eq:fp} implies that 
\[
\nor{G(x)-x_*}\leq\nor{x-x_*}< r,
\]
i.e. $G(x)\in B_r(x_*)$. Thus $G(B_r(x_*))\subseteq B_r(x_*)$ and a sequence can be defined in $B_r(x_*)$ by choosing $x_0\in B_{r}(x_*)$ and setting $x_{n+1}=G(x_n)$.  Being $\nor{x_n-x_*}<\bar{r}$, again from the definition of $\bar{r}$ we have  $q(\nor{x_n-x_*})<1$, and from \eqref{eq:fp}  we get

\[\nor{x_{n+1}-x}=\nor{G(x_{n})-x_*}\leq q(\nor{x_n-x_*})\nor{x_n-x_*}<\nor{x_n-x_*}.\]
This implies that $q(\nor{x_{n+1}-x_* })\leq q(\nor{x_n-x_*})$, since $q$ is increasing. Therefore, denoting $q(\nor{x_0-x_*})=:q_0$ we get $q(\nor{x_n-x_*})\leq q_0<1$ for all $n\in\mathbb{N}$ and 
\[
\nor{x_{n+1}-x_*}\leq  q(\nor{x_n-x_*})\nor{x_n-x_*}\leq q_0\nor{x_n-x_*}\leq\ldots\leq q_0^n\nor{x_0-x^*}.
\]
\end{proof}
We now introduce some notations, allowing for rewriting the conditions which have been described in Proposition \ref{prop:ex} for a local minimizer of $\Phi_0$.  Define $G$ and $\tilde{G}$ by setting
\beeq{eq:geg}
{G(x) = x - F^\prime(x)^\dag F(x)\qquad\text{and}\qquad \tilde{G}(x) = \mathrm{prox}_{J}^{H(x)}( G(x)),}
where $H(x)=F'(x)^*F'(x)$.   The domain of $G$ and $\tilde{G}$  is the subset $\dd$ of $\Omega$ defined as 
\beeq{eq:defd}{\dd=\{x\in\Omega\,:\, \text{$F'(x)$ is injective and $R(F'(x))$ is closed}\}.}

If $x_*\in\dd$ is a local minimizer of ($\mathcal{P}_0$) the fixed point equation of Proposition \ref{prop:ex} can be restated by saying that $x_*$ is a fixed point for $\tilde G$, namely
\beeq{eq:fix}
{x_*=\tilde G(x_*).
}

\begin{proposition}\label{prop:gtilde} Assume $(SH)$ and let $x_*$ be a local  minimizer of $\Phi_0$ belonging to $\dd$. Suppose  that $U\subseteq \Omega$ is open and starshaped with respect to $x_*$. Moreover assume

\begin{itemize}
\item[i)] $F^\prime:\Omega\subseteq X \to \mathbb{L}(X,Y)$ is center Lipschitz continuous  of center $x_*$ with $L:[0,R)\to\mathbb{R}_+$ average on $U\subseteq\Omega$;
\item[ii)] $\alpha=\nor{F(x_*)}$, $\beta=\nor{F'(x_*)^\dag}$ and $\kappa=\nor{F'(x_*)^{\dag}}\nor{F'(x_*)}$.
\end{itemize} 
Then, defining  $\bar{R}$ by setting
\[
\bar{R}=\sup\{r\in (0,R)\,:\, \gamma_0(r)r<1/\beta\}
\]
it follows that for all  $r\in\rone$ with $0<r\leq \bar{R}$ and $B_r(x_*)\subseteq U$, $\tilde{G}$ satisfies
\[
\nor{\tilde{G}(x)-x_*}\leq q(\nor{x-x_*})\nor{x-x_*},  
\]
for all $x\in B_r(x_*)$, where $q:[0,\bar{R})\to\rone_+$ is defined as 
\begin{align*}
q(r) &= \frac{ \beta}{1 - \beta \gamma_0(r) r} \bigg\{ \frac{ \beta \gamma_0(r) \gamma^c(r) r^2+ \kappa \gamma^c(r) r}{(1 - \beta \gamma_0(r) r)} + \frac{(1 + \sqrt{2}) \alpha \beta^2 \gamma_0(r)^2 r }{1 - \beta \gamma_0(r) r} \\[1.5ex]
&\qquad \qquad\qquad \qquad\qquad \qquad\qquad \qquad+ \frac{[(1 + \sqrt{2})\kappa + 1] \alpha \beta \gamma_0(r)}{1 - \beta \gamma_0(r) r}  \bigg\}
\end{align*}
and it is continuous and strictly increasing.
 \end{proposition}
\begin{proof}
Since $x_*$ is a local minimizer of $\Phi_0$ and $x_*\in\dd$, with $\dd$ defined as in \eqref{eq:defd}, $x_*$ is a fixed point of $\tilde{G}$ (see \eqref{eq:fix}), therefore
\beeq{eq:fixcon}
{G(x_*) - \mathrm{prox}_{J}^{H(x_*)} G(x_*) = G(x_*) - x_* = F^\prime(x_*)^\dag F(x_*).}
Fix $r\in\rone$, with $0<r\leq \bar{R}$ such that $B_r(x_*)\subseteq U$ and take $x\in B_r(x_*)$. Adopting the notations of Proposition \ref{prop:gam}, as noted in Remark \ref{rem:gam}, from  the center Lipschitz hypothesis we get
\[
\nor{F'(x)-F'(x_*)}\nor{F'(x_*)^\dag}\leq \gamma_0(\nor{x-x_*})\nor{x-x_*}\beta.
\]
Recalling that Remark \ref{rem:gam} ensures that the function $\rho\mapsto \rho\gamma_0(\rho)$  is continuous, strictly increasing and  takes value 0 in 0, we have that $\bar{R}>0$ and 
\[
\nor{F'(x)-F'(x_*)}\nor{F'(x_*)^\dag}<\beta\gamma_0(r)r\leq 1, 
\]
 and thus applying Lemma \ref{lem:daga}, $F'(x)$ is injective, with closed range, and 
 \beeq{eq:bdag}
{ \nor{F'(x)^\dag}\leq \frac{\beta}{1-\beta\gamma_0(\rho_x)\rho_x }, \quad\text{where }\quad\rho_x=\nor{x-x_*}.  }
Applying inequality \eqref{eq:vartut} with $H_1=H(x)$, $H_2=H(x_*)$, $z_1=G(x)$ and $z_2=G(x_*)$, and taking into account \eqref{eq:fixcon} we get

\begin{align}
\nonumber\nor{\tilde{G}(x) - x_*}& = \nor{\mathrm{prox}_{J}^{H(x)}( G(x)) - \mathrm{prox}_{J}^{H(x_*)}( G(x_*))}\\[1ex]
\nonumber&\leq \big(\nor{H(x)} \nor{H(x)^{-1}}\big)^{1/2} \nor{G(x) - G(x_*)} \\[1ex]
\nonumber&\quad+ \nor{H(x)^{-1}}\big\lVert(H(x) - H(x_*))\big(G(x_*) - \mathrm{prox}_{J}^{H(x_*)} G(x_*)\big) \big\rVert \\[1ex]
\nonumber& = \big(\nor{H(x)} \nor{H(x)^{-1}}\big)^{1/2} \nor{G(x) - G(x_*)} \\[1ex]
\label{eq:Guno}&\quad+ \nor{H(x)^{-1}} \nor{(H(x) - H(x_*))F^\prime(x_*)^\dag F(x_*)}.
\end{align}
Moreover
\begin{align*}
\nor{H(x)} &= \nor{F^\prime(x)^*F^\prime(x)} = \nor{F^\prime(x)}^2\\[1ex]
\nor{H(x)^{-1}} &= \nor{[F^\prime(x)^*F^\prime(x)]^{-1}} = \nor{F^\prime(x)^\dag}^2.
\end{align*}
Recalling the properties of the Moore-Penrose generalized inverse in \eqref{lem:gids2} and Lemma \ref{lem:daga} we get
\begin{align*}
(H(x) - H(x_*))F^\prime(x_*)^\dag &= (F^\prime(x)^* F^\prime(x) - F^\prime(x_*)^* F^\prime(x_*))F^\prime(x_*)^\dag \\
&= F^\prime(x)^* F^\prime(x) F^\prime(x_*)^\dag - F^\prime(x_*)^* F^\prime(x_*) F^\prime(x_*)^\dag \\
&= F^\prime(x)^* F^\prime(x) F^\prime(x_*)^\dag - F^\prime(x)^*P_{R(F^\prime(x_*))} \\& \qquad\quad+ F^\prime(x)^*P_{R(F^\prime(x_*))} - F^\prime(x_*)^* P_{R(F^\prime(x_*))} \\
& = F^\prime(x)^*(F^\prime(x) - F^\prime(x_*))F^\prime(x_*)^\dag \\ & \qquad\quad+ (F^\prime(x) - F^\prime(x_*))^*P_{R(F^\prime(x_*)),}
\end{align*}
therefore, 
\begin{equation}\label{eq:Gdue}
\nor{(H(x) - H(x_*))F^\prime(x_*)^\dag  } \leq \left( \nor{F^\prime(x)} \nor{F^\prime(x_*)^\dag} + 1\right)\nor{F^\prime(x) - F^\prime(x_*)}.
\end{equation}
Hence, substituting in \eqref{eq:Guno} the bound derived in \eqref{eq:Gdue} we obtain
\begin{align}\label{eq:Gdue2}
\nor{\tilde{G}(x) - x_*} &\leq \nor{F^\prime(x)}\nor{F^\prime(x)^\dag} \nor{G(x) - G(x_*)} \\[1ex]
\nonumber& \quad+ \nor{F^\prime(x)^ \dag}^2 \left( \nor{F^\prime(x)} \nor{F^\prime(x_*)^\dag} + 1\right)\nor{F^\prime(x) - F^\prime(x_*)}\nor{ F(x_*)}.
\end{align}
On the other hand, thanks to the properties of the Moore-Penrose pseudoinverse reported in \eqref{lem:gids2} and the injectivity of $F'(x)$
\begin{align*}
G(x) - G(x_*)&= x - x_* - F^\prime(x)^\dag F(x) + F^\prime(x_*)^\dag F(x_*) \\
&= F^\prime(x)^\dag [F^\prime(x)(x - x_*) - F(x) + F(x_*)] \\
& \quad + (F^\prime(x_*)^\dag - F^\prime(x)^\dag) F(x_*)
\end{align*}
and thus, using Lemma \ref{lem:daga}
\begin{align}
\nonumber\nor{G(x) - G(x_*)} &\leq \nor{F^\prime(x)^\dag} \nor{F(x_*) - F(x) - F^\prime(x)(x_* - x)}\\[1ex]
\nonumber& \qquad+ \nor{F^\prime(x_*)^\dag - F^\prime(x)^\dag} \nor{F(x_*)}\\[1ex]
\nonumber&\leq  \nor{F^\prime(x)^\dag} \nor{F(x_*) - F(x) - F^\prime(x)(x_* - x)}\\[1ex]
\nonumber&\qquad + \sqrt{2}\nor{F^\prime(x)^\dag} \nor{F^\prime(x_*)^\dag} \nor{F^\prime(x) - F^\prime(x_*)} \nor{F(x_*)} \\[1ex]
\nonumber& = \nor{F^\prime(x)^\dag} \big\{ \nor{F(x_*) - F(x) - F^\prime(x)(x_* - x)}\\[1ex]
\label{eq:Gtre}&\qquad + \sqrt{2} \nor{F^\prime(x_*)^\dag} \nor{F(x_*)} \nor{F^\prime(x) - F^\prime(x_*)} \big\}
\end{align}
Substituting  \eqref{eq:Gtre} in \eqref{eq:Gdue2}
\begin{align}
\nonumber\nor{\tilde{G}(x) - x_*} &\leq  \nor{F^\prime(x)^\dag}^2 \Big\{ \nor{F^\prime(x)} \big( \nor{F(x_*) - F(x) - F^\prime(x)(x_* - x)} \\[1ex]
\label{eq:qq}& \qquad  + (1+\sqrt{2}) \nor{F^\prime(x_*)^\dag} \nor{F(x_*)} \nor{F^\prime(x) - F^\prime(x_*)} \big)\\
\nonumber& \left.\qquad  + \nor{F^\prime(x) - F^\prime(x_*)} \nor{F(x_*)} \right]\Big\}
\end{align}
Taking into account Remark \ref{rem:gam}, we can  rewrite inequality \eqref{eq:qq} as
\begin{align}
&\nonumber\nor{\tilde{G}(x) - x_*}\\ \nonumber&\leq \nor{F'(x)^\dag}^2\left\{\nor{F'(x)} \left[\gamma^c(\rho_x)\rho_x^2+(1+\sqrt{2})\|F'(x_*)^\dag\|\nor{F(x_*)}\gamma_0(\rho_x)\rho_x\right ]\right. \\
\nonumber&\quad+\nor{F(x_*)}\gamma_0(\rho_x)\rho_x \Big\}&
\end{align}
To find a  bound for the quantity $\nor{F'(x)}$, recall that $\kappa$ is the conditioning number of $F'(x_*)$, i.e. $\kappa:=\nor{F'(x_*)^\dag}\nor{F'(x_*)}$, and by the triangular inequality and Remark \eqref{rem:gam} we get $\nor{F'(x)}\leq\nor{F'(x)-F'(x_*)}+\nor{F'(x_*)}\leq \gamma_0(\rho_x)\rho_x+\kappa/\beta$.  Thus, recalling \eqref{eq:bdag}, we finally obtain
\begin{align}
&\nonumber\nor{\tilde{G}(x) - x_*} \\ \nonumber&\leq\frac{\beta^2}{\left(1-\beta\gamma_0(\rho_x)\rho_x\right)^2 }\Big\{(\gamma_0(\rho_x)\rho_x+\kappa/\beta) \Big[\gamma^c(\rho_x)\rho_x^2+(1+\sqrt{2})\beta\alpha\gamma_0(\rho_x)\rho_x\Big]  \\
\nonumber&\hspace{9cm}+\alpha\gamma_0(\rho_x)\rho_x \Big\},
\end{align}
or, equivalently
\begin{align*}
\nor{\tilde{G}(x) - x_*} &\leq\frac{\beta}{1-\beta\gamma_0(\rho_x)\rho_x }\left\{ \frac{ \beta \gamma_0(\rho_x) \gamma^c(\rho_x) \rho_x^2 + \kappa \gamma^c(\rho_x) \rho_x}{1-\beta\gamma_0(\rho_x)\rho_x } \right.\\[1.5ex]  
&\qquad \left.+\frac{(1 + \sqrt{2}) \alpha \beta^2 \gamma_0(\rho_x)^2 \rho_x}{1-\beta\gamma_0(\rho_x)\rho_x }  +\frac{[(1+\sqrt{2}) \kappa+1] \alpha \beta \gamma_0(\rho_x) }{1-\beta\gamma_0(\rho_x)\rho_x }   \right\}\rho_x \\[1.5ex]
&=q(\nor{x-x_*})\nor{x-x_*},
\end{align*}
where we set 
\begin{align*}
q(r) &= \frac{ \beta}{1 - \beta \gamma_0(r) r} \bigg\{ \frac{ \beta \gamma_0(r) \gamma^c(r) r^2+ \kappa \gamma^c(r) r}{(1 - \beta \gamma_0(r) r)} + \frac{(1 + \sqrt{2}) \alpha \beta^2 \gamma_0(r)^2 r }{1 - \beta \gamma_0(r) r} \\[1.5ex]
&\qquad \qquad\qquad \qquad\qquad \qquad\qquad \qquad+ \frac{[(1 + \sqrt{2})\kappa + 1] \alpha \beta \gamma_0(r)}{1 - \beta \gamma_0(r) r}  \bigg\}
\end{align*}
Finally, it is easy to prove that $q$ is continuous and strictly increasing relying on Remark \ref{rem:gam}. 
\end{proof}

\ \\

\noindent{\em Proof of Theorem \ref{thm:1}}

Let $G$ and $\tilde{G}$ be defined as in \eqref{eq:geg} and define $\bar{R}$ and $q$ as in Proposition \ref{prop:gtilde}. Now fix $r< \bar{r}$ such that $B_r(x_*)\subseteq U$. Then, thanks to hypothesis 3), it is possible to apply Proposition \ref{prop:fp} and to get the first part of the thesis. 
Finally, relying on the structure of the function $q$ shown in Proposition \ref{prop:gtilde}, and denoting $\rho_{x_0}=\nor{x-x_0}$,  the expression of the constants $C_1, C_2$  easily follows.

\section{Applications}\label{sec:app}

Algorithm \eqref{eq:xn1} is a two-steps algorithm, consisting of the classical Gauss-Newton step followed by a ``$J$-projection''  in a variable metric. In this section the general framework shall be specialized to solve constrained nonlinear systems of equations in the least squares sense. We remark that due to hypotheses of Theorem \ref{thm:1}  we are dealing with regular problems in the sense of Bakushinski\u{\i} and Kokurin \cite{BakKok04}. In the finite dimensional case this implies the number of equations to be greater than the number of unknowns.
This subject has been studied  in  \cite{BelMacMor04,KozMarSan97,Kan01,Ulb01} (see also references therein).  
Denoting by $C$ a closed and convex subset of $X$, we consider the problem
\beeq{eq:Pc}
{\min_{x\in C} \nor{F(x)}^2,}
which can be cast in our framework by setting $J(x)=\iota_C(x)$, where $\iota_C$ denotes the indicator function of the set $C$, i.e. $\iota_C(x)=0$ for $x\in C$ and $\iota_C(x)=+\infty$ otherwise. The proximity operator of $\iota_C$ with respect to $H(x_n)=F'(x_n)^*F'(x_n)$, turns out to be the projection onto $C$ w.r.t. the metric defined by $H(x_n)$, and therefore algorithm \eqref{eq:xn1} reads as follows
\beeq{eq:xnP}{
x_{n+1}=P_C^{H(x_n)} (x_n-F'(x_n)^\dag F(x_n)).}
Since in general  a closed form of the projection operator is not available, a further algorithm is needed for solving the projection task, which adds an inner iteration to the main procedure. We choose the {\em forward-backward algorithm} \cite{ComWaj05}. Though there are many other methods for that purpose, we do not carry out any comparison among them, because this is beyond the scope of the present paper.   
By definition of proximity operator, given $H=A^*A$, $A$ injective with closed range, we have
\begin{align*}
P_{C}^H(z)&=\prox_{\iota_C}^H(z)\\
&=\argmin\Big\{\iota_C(v)+\frac 1 2 \nor{v-z}^2_H\Big \}\\
&=\argmin \Big \{\iota_C(v)+\frac 1 2 \nor{Av-Az}^2 \Big\}.
\end{align*}
If $P_C$ denotes the projection onto the convex set $C$,  now with respect to the original metric of the space $X$,  the sequence defined by
\begin{align*}
v_0&\in X\\
v_{k+1}&=P_{C}(v_k-\sigma H(v_k-z)),
\end{align*}
with $\sigma\leq {2}/{\nor{A}^2}$, is strongly convergent to the point $P_{C}^H(z)$ \cite{rosasco2010ecml}. Eventually, the  full algorithm is
\begin{equation}\label{eq:cgn}\left[\begin{aligned}
&x_0\in C\\
&z_n=x_n-F'(x_n)^\dag F(x_n)\\[1.2ex]
&\quad \left[ \begin{aligned}& v_{0,n}\in C,\  \ \sigma_n\leq {1}/{2\nor{F'(x_n)}^2} \\[.7ex]& v_{k+1,n}=P_C(v_{k,n}-\sigma_n F'(x_n)^*F'(x_n)(v_{k,n}-z_n)) \end{aligned} \right.\\[1.2ex] 
&x_{n+1}=\lim_{k}v_{k,n}.
\end{aligned} \right.\end{equation}
It is worth noting that the inner iteration is not required when $z_n$ belongs to $C$. Indeed, in that case,  the projection leaves $z_n$  untouched and the full step of the algorithm \eqref{eq:xnP} reduces to the classical Gauss-Newton step. Such situation asymptotically occurs when $x_*$ is internal to $C$.

Algorithm \eqref{eq:cgn} requires explicit evaluation of the projection $P_C$, which can be  done for simple sets, like spheres, boxes, etc. Particularly relevant from the point of view of the applications --- see \cite{BelMacMor04} and references therein ---  is the case of box constraints in $\mathbb{R}^n$.    
  We point out that when $P_C$ can be computed explicitly the algorithm generates a sequence of feasible points, no matter when the inner iteration is stopped. This feature can be useful in forcing the sequence of iterates to remain in regions where the function is well-behaved, avoiding the Gauss-Newton step to lead to sites where the derivative is ill-conditioned.   

\section{Numerical experiments}\label{sec:num}

This section summarizes the results of the numerical experiments we carried out in order to verify the effectiveness of algorithm \eqref{eq:cgn} for solving real-life constrained nonlinear least-squares problems. In particular, we consider the case of box constraints, namely
\[
\min_{x\in \mathbb{R}^n} \nor{F(x)-y}^2,\qquad a\leq x\leq b,
\] 
where $a$ and $b$ are in $\overline{\mathbb{R}}^n$, $a\leq b$ and $C=\prod_{i=1}^n [a_i,b_i]$, $F:\Omega\supseteq C\to \mathbb{R}^m$.

The aim of the tests is to illustrate the behavior of our algorithm on some representative examples and  show that it can be successfully applied to real problems.  
The algorithm is implemented in MatLab, and the convergence tests
\[
\nor{x_{n+1}-x_n}<\epsilon, \quad \nor{v_{k+1,n}-v_{k,n}}<\epsilon
\]
 are used with precision $\epsilon=10^{-12}$. 

We remark that the implementation of the algorithm \eqref{eq:cgn} computes the projection only approximately, meaning that the internal iteration is stopped either because the required precision has been attained or because an a priori fixed maximum number of iterations has been reached. For this reason, the projection step depends on the algorithm selected to that aim and the forward-backward algorithm is just one choice among several possibilities. Furthermore, the number of evaluations of $F$ and $F'$ depends only on the number of outer iterations. Therefore we provide just the number of outer iterations needed to reach the target precision as a  measure of our method's performance.  
Yet,  the number of inner iterations does affect the number of outer iterations. In fact, in our experiments we  observed  that,  even though the algorithm is quite robust with respect to errors in the computation of the projection, the number of outer iterations can increase if the required inner precision is not attained (inner iterations reach the maximum allowed). 

The experiments are performed on some standard small residual test problems. One group of them is taken from \cite{Fle87} and a second group comes from the extensive library NLE \cite{ShaBraCut02}, which is accessible through the web site: \texttt{www.polymath-software.com/library}. 
 We considered only problems for which the solution (or a good estimate of it) is known in advance and for which  the Gauss-Newton method is known to be effective --- since our proposal in fact extends the classical one.
The problems we select in the first group are \texttt{Rosenbrock}, \texttt{Osborne1}, \texttt{Osborne2} \cite{Lootsma} and \texttt{Kowalik} \cite{KowOsb69}. They are actually unconstrained problems, to which we added  some box constraints set up in order to make the  provided solution fall on the boundary  of the box (faces, edges, vertices, etc.). Besides, on Rosenbrock's example we tried out our method also in case the global minimizer is kept outside the fixed box.

The remaining problems, \texttt{Twoeq6} and \texttt{Teneq1b}, come as truly constrained and are labeled as ``higher difficulty level'' in the NLE library.  Unlike the first group, here the constrained minimizers of \texttt{Twoeq6} and \texttt{Teneq1b} lie in the interior of the feasible set. 
 Observe in addition that the given constraints define a convex set that is not closed. More specifically, in \texttt{Teneq1b} example, the feasible region is the positive orthant of $\mathbb{R}^{10}$, excluding four coordinate hyperplanes where the first derivative is undefined. We overcome this difficulty by shrinking slightly the feasible region of a small amount $\delta$ 
\[
C_{\delta}=\Big\{(x_1,\ldots,x_{10})\,:\, x_i\geq \delta \text{ for } i=1,\ldots,4,\ \  x_i\geq 0 \text{ for } i\geq 5\Big\},
\]   
and solving the problem in the closed and convex set $C_\delta$ with $\delta=0.0001$. The same trick has been used also for solving problem \texttt{Twoeq6} too.

   The experimental protocol is different for the two test groups: in the former group, 20 points belonging to the box are randomly chosen as initial guesses and the average number of required iterations is provided; whereas in the latter  one  the algorithm is fed with two critical initializations reported in the NLE's problem description. The results are collected in Table \ref{tab:one}.

In all tests the algorithm reached the solution up to the required precision, and we did not detect any significant variation in the number of iterations depending on the starting points. Along the path we checked the condition number of $F',$ observing that it always keeps bounded from above, the problem thus remaining well-conditioned.   
In case of \texttt{Osborne2}, we saw that the classical Gauss-Newton method does not converge for some initializations, due to ill-conditioning of the derivative $F'$.  We were able to correct this behavior by setting the constraints properly around the known minimizer.

\section{Conclusions} 

This paper shows that the local theory on the convergence of the Gauss-Newton method can be extended to deal with the more general case of least squares problems with a convex penalty. The main theoretical result demonstrates that, under weak Lipschitz conditions on the derivative, convergence rates analogous to those existing for the standard case can be derived. An explicit formula for the radius of the convergence ball is also provided. A valuable application we propose concerns nonlinear equations with constraints. Our algorithm has been found effective and robust in solving such problems as shown in several numerical tests. Both the cases of solutions on the boundary of the feasible set as well as solutions in its interior as been treated successfully.\\

\noindent \textbf{Acknowledgements }
 We are grateful to Alessandro Verri for his constant support and advice.
We further thank Curzio Basso for carefully reading our paper.

\begin{center}
{\small
\begin{table}[h]
\caption{Results of numerical experiments on the test problems specified in the first column for the box constraints given by $a$ and $b$ with starting point $x_0$. The point $x_*$ is the detected minimizer, which we show with five digits precision for conciseness.}
\label{tab:one}       
\begin{tabular}[h]{l|c|c|c|c|c|c|c}
Function & $n$ & $m$ & $a$&$b$&$x_0$&$x_*$& avg. n. iter.\\
\noalign{\smallskip}\hline\hline\noalign{\smallskip}
\texttt{Rosenbrock} & 2& 2  &$\begin{bmatrix}  -3\\-2 \end{bmatrix}$& $\begin{bmatrix}  3\\0.8 \end{bmatrix}$ &$20$ random&$\begin{bmatrix}  0.89475\\0.80000 \end{bmatrix}$& 7\\
&&&&&&&\\
\texttt{Kowalik}&4&11&$\begin{bmatrix}0.1928\\0.1916\\0.1234\\0.1362\end{bmatrix}$&$\begin{bmatrix}1\\1\\1\\1\end{bmatrix}$&$20$ random&$\begin{bmatrix}0.19281\\0.19165\\0.12340\\0.13620\end{bmatrix}$&7\\
&&&&&&&\\
\texttt{Osborne1} &5&31 &$\begin{bmatrix}0.3754\\ 1\\-2\\0.01287\\0 \end{bmatrix}$ & $\begin{bmatrix}1\\ 2\\0\\1\\1 \end{bmatrix}$&$20$ random&$ \begin{bmatrix}0.37546\\1.93569\\-1.46461\\0.01287\\0.02212\end{bmatrix}$&$21$\\
&&&&&&&\\
\texttt{Osborne2}&11&65&$\begin{bmatrix} 1.31\\0.4314\\0.6336\\0.5\\0.5\\0.6\\1\\4\\2\\4.5689\\5\end{bmatrix}$&$\begin{bmatrix}1.4\\0.8\\1\\1\\1\\3\\5\\7\\2.5\\5\\6\end{bmatrix}$&$20$ random&$\begin{bmatrix}1.31000\\ 0.43157\\ 0.63367 \\0.59941\\0.75423\\0.90423\\1.36573\\4.82393\\2.39867\\4.56890\\5.67535 \end{bmatrix}$&17\\
&&&&&&&\\
\texttt{Twoeq6}&2&2&$\begin{bmatrix}  0.0001\\ 0.0001 \end{bmatrix}$  &$\begin{bmatrix}  0.9999\\ +\infty \end{bmatrix}$&$\begin{bmatrix}0.9 \\ 0.5 \end{bmatrix}\begin{bmatrix}0.6 \\ 0.1 \end{bmatrix}$&$\begin{bmatrix}0.75739\\ 0.02130\end{bmatrix}$&20\\
&&&&&&&\\
\texttt{Teneq1b} & 10&10&$\begin{bmatrix} 0.0001\\ 0.0001\\ 0.0001\\ 0.0001 \\0 \\0 \\0 \\0 \\0 \\ 0   \end{bmatrix}$& $\begin{bmatrix}  +\infty\\ +\infty\\+\infty\\+\infty\\+\infty\\+\infty\\+\infty\\+\infty\\+\infty\\+\infty\end{bmatrix}$&$ \begin{bmatrix}  1\\ 1\\20\\1\\0\\0\\0\\0\\0\\1 \end{bmatrix}\begin{bmatrix}2\\5\\40\\1\\0\\0\\0\\0\\0\\5\end{bmatrix}$&
$\begin{bmatrix} 2.99763\\ 3.96642\\ 79.99969\\ 0.00236\\ 0.00060\\ 0.00136\\0.06457\\ 3.53081\\ 26.43154\\ 0.00449 \end{bmatrix}$&
10

\end{tabular}
\end{table}}
\end{center}



\begin{thebibliography}{10}
\providecommand{\url}[1]{{#1}}
\providecommand{\urlprefix}{URL }
\expandafter\ifx\csname urlstyle\endcsname\relax
  \providecommand{\doi}[1]{DOI~\discretionary{}{}{}#1}\else
  \providecommand{\doi}{DOI~\discretionary{}{}{}\begingroup
  \urlstyle{rm}\Url}\fi

\bibitem{ArgHil09}
Argyros, I., Hilout, S.: On the local convergence of the {G}auss-{N}ewton
  method.
\newblock Punjab Univ. J. Math. (Lahore) \textbf{41}, 23--33 (2009)

\bibitem{ArgHil10}
Argyros, I., Hilout, S.: On the {G}auss�{N}ewton method.
\newblock Journal of Applied Mathematics and Computing pp. 1--14 (2010)

\bibitem{Arg05}
Argyros, I.K.: On the semilocal convergence of the {G}auss-{N}ewton method.
\newblock Adv. Nonlinear Var. Inequal. \textbf{8}(2), 93--99 (2005)

\bibitem{BacBur09}
Bachmayr, M., Burger, M.: Iterative total variation schemes for nonlinear
  inverse problems.
\newblock Inverse Problems \textbf{25}(10), 105,004, 26 (2009)

\bibitem{BakKok04}
Bakushinsky, A.B., Kokurin, M.Y.: Iterative methods for approximate solution of
  inverse problems, \emph{Mathematics and Its Applications (New York)}, vol.
  577.
\newblock Springer, Dordrecht (2004)

\bibitem{BelMacMor04}
Bellavia, S., Macconi, M., Morini, B.: S{TRSCNE}: a scaled trust-region solver
  for constrained nonlinear equations.
\newblock Comput. Optim. Appl. \textbf{28}(1), 31--50 (2004)

\bibitem{Ben65}
Ben-Israel, A.: A modified {N}ewton-{R}aphson method for the solution of
  systems of equations.
\newblock Israel J. Math. \textbf{3}, 94--98 (1965)

\bibitem{BlaNeuSch97}
Blaschke, B., Neubauer, A., Scherzer, O.: On convergence rates for the
  iteratively regularized {G}auss-{N}ewton method.
\newblock IMA J. Numer. Anal. \textbf{17}(3), 421--436 (1997)

\bibitem{borwein99}
Borwein, J.M., Lewis, A.S.: Convex Analysis and Nonlinear Optimization.
\newblock Advanced Books in Mathematics. Canadian Mathematical Society (2000)

\bibitem{BurFer95}
Burke, J.V., Ferris, M.C.: A {G}auss-{N}ewton method for convex composite
  optimization.
\newblock Math. Programming \textbf{71}(2, Ser. A), 179--194 (1995)

\bibitem{ComWaj05}
Combettes, P.L., Wajs, V.R.: Signal recovery by proximal forward-backward
  splitting.
\newblock Multiscale Model. Simul. \textbf{4}(4), 1168--1200 (electronic)
  (2005)

\bibitem{DenSch96}
Dennis Jr., J.E., Schnabel, R.B.: Numerical methods for unconstrained
  optimization and nonlinear equations, \emph{Classics in Applied Mathematics},
  vol.~16.
\newblock Society for Industrial and Applied Mathematics (SIAM), Philadelphia,
  PA (1996).
\newblock Corrected reprint of the 1983 original

\bibitem{EngHanNeu96}
Engl, H.W., Hanke, M., Neubauer, A.: Regularization of inverse problems.
\newblock Mathematics and its Applications \textbf{375} (1996)

\bibitem{FerSva09}
Ferreira, O.P., Svaiter, B.F.: Kantorovich's majorants principle for {N}ewton's
  method.
\newblock Comput. Optim. Appl. \textbf{42}(2), 213--229 (2009)

\bibitem{Fle87}
Fletcher, R.: Practical methods of optimization.
\newblock John Wiley and sons, West Sussex, England (1987)

\bibitem{FloPar90}
Floudas, C.A., Pardalos, P.M.: A collection of test problems for constrained
  global optimization algorithms, \emph{Lecture Notes in Computer Science},
  vol. 455.
\newblock Springer-Verlag, Berlin (1990)

\bibitem{Gro77}
Groetsch, C.W.: Generalized inverses of linear operators: representation and
  approximation.
\newblock Marcel Dekker Inc., New York (1977).
\newblock Monographs and Textbooks in Pure and Applied Mathematics, No. 37

\bibitem{Hau86}
H{\"a}ussler, W.M.: A {K}antorovich-type convergence analysis for the
  {G}auss-{N}ewton-method.
\newblock Numer. Math. \textbf{48}(1), 119--125 (1986)

\bibitem{HirLem93}
Hiriart-Urruty, J.B., Lemar{\'e}chal, C.: Convex analysis and minimization
  algorithms. {II}, \emph{Grundlehren der Mathematischen Wissenschaften
  [Fundamental Principles of Mathematical Sciences]}, vol. 306.
\newblock Springer-Verlag, Berlin (1993)

\bibitem{Jin08}
Jin, Q.: A convergence analysis of the iteratively regularized {G}auss-{N}ewton
  method under the {L}ipschitz condition.
\newblock Inverse Problems \textbf{24}(4), 045,002, 16 (2008)

\bibitem{KalHof10}
Kaltenbacher, B., Hofmann, B.: Convergence rates for the iteratively
  regularized {G}auss-{N}ewton method in {B}anach spaces.
\newblock Inverse Problems \textbf{26}(3), 035,007, 21 (2010)

\bibitem{Kan01}
Kanzow, C.: An active set-type {N}ewton method for constrained nonlinear
  systems.
\newblock Appl. Optim. \textbf{50}, 179�200 (2001)

\bibitem{KowOsb69}
Kowalik, J., Osborne, M.R.: Methods for unconstrained Optimization Problems.
\newblock American Elsevier (1969)

\bibitem{KozMarSan97}
Kozakevich, D.N., Mart{\'{\i}}nez, J.M., Santos, S.A.: Solving nonlinear
  systems of equations with simple constraints.
\newblock Mat. Apl. Comput. \textbf{16}(3), 215--235 (1997)

\bibitem{Lan10}
Langer, S.: Investigation of preconditioning techniques for the iteratively
  regularized {G}auss-{N}ewton method for exponentially ill-posed problems.
\newblock SIAM J. Sci. Comput. \textbf{32}(5), 2543--2559 (2010)

\bibitem{LewWri08}
Lewis, A., Wright, S.J.: A proximal method for composite optimization (2008).
\newblock \urlprefix\url{http://arxiv.org/abs/0812.0423}

\bibitem{LiHuWan10}
Li, C., Hu, N., Wang, J.: Convergence behavior of {G}auss-{N}ewton's method and
  extensions of the {S}male point estimate theory.
\newblock J. Complexity \textbf{26}(3), 268--295 (2010)

\bibitem{LiNg07}
Li, C., Ng, K.F.: Majorizing functions and convergence of the {G}auss-{N}ewton
  method for convex composite optimization.
\newblock SIAM J. Optim. \textbf{18}(2), 613--642 (electronic) (2007)

\bibitem{LiWan02}
Li, C., Wang, X.: On convergence of the {G}auss-{N}ewton method for convex
  composite optimization.
\newblock Math. Programming \textbf{91}(2, Ser. A), 349--356 (2002)

\bibitem{LiWan03}
Li, C., Wang, X.: Convergence of {N}ewton's method and uniqueness of the
  solution of equations in {B}anach spaces. {II}.
\newblock Acta Math. Sin. (Engl. Ser.) \textbf{19}(2), 405--412 (2003)

\bibitem{LiZhaJin04}
Li, C., Zhang, W., Jin, X.: Convergence and uniqueness properties of
  {G}auss-{N}ewton's method.
\newblock Comput. Math. Appl. \textbf{47}(6-7), 1057--1067 (2004)

\bibitem{Lootsma}
Lootsma, F. (ed.): Numerical methods for nonlinear optimization.
\newblock Academic Press Inc., U.S. (1972)

\bibitem{Mor62}
Moreau, J.J.: Fonctions convexes duales et points proximaux dans un espace
  hilbertien.
\newblock C. R. Acad. Sci. Paris \textbf{255}, 2897--2899 (1962)

\bibitem{Mor63}
Moreau, J.J.: Propri\'et\'es des applications ``prox''.
\newblock C. R. Acad. Sci. Paris \textbf{256}, 1069--1071 (1963)

\bibitem{Mor65}
Moreau, J.J.: Proximit\'e et dualit\'e dans un espace hilbertien.
\newblock Bull. Soc. Math. France \textbf{93}, 273--299 (1965)

\bibitem{Pol83}
Polyak, B.T.: Introduction to optimization.
\newblock Translations Series in Mathematics and Engineering. Optimization
  Software Inc. Publications Division, New York (1987).
\newblock Translated from the Russian, With a foreword by Dimitri P. Bertsekas

\bibitem{RamTes05}
Ramlau, R., Teschke, G.: Tikhonov replacement functionals for iteratively
  solving nonlinear operator equations.
\newblock Inverse Problems \textbf{21}(5), 1571--1592 (2005)

\bibitem{RamTes06}
Ramlau, R., Teschke, G.: A {T}ikhonov-based projection iteration for nonlinear
  ill-posed problems with sparsity constraints.
\newblock Numer. Math. \textbf{104}(2), 177--203 (2006)

\bibitem{rosasco2010ecml}
Rosasco, L., Mosci, S., Santoro, M., Verri, A., Villa, S.: Proximal methods for
  structured sparsity regularization.
\newblock LNAI, Springer (to appear)  (2010)

\bibitem{SchGraGro09}
Scherzer, O., Grasmair, M., Grossauer, H., Haltmeier, M., Lenzen, F.:
  Variational methods in imaging, \emph{Applied Mathematical Sciences}, vol.
  167.
\newblock Springer, New York (2009)

\bibitem{ShaBraCut02}
Shacham, M., Brauner, N., Cutlib, M.: A web-based library for testing
  performance of numerical software for solving nonlinear algebraic equations.
\newblock Comp \& Chem. Eng. \textbf{26}, 547--554 (2002)

\bibitem{Ste69}
Stewart, G.W.: On the continuity of the generalized inverse.
\newblock SIAM J. Appl. Math. \textbf{17}, 33--45 (1969)

\bibitem{Ulb01}
Ulbrich, M.: Nonmonotone trust-region methods for bound-constrained semismooth
  equations with applications to nonlinear mixed complementarity problems.
\newblock SIAM J. Optim. \textbf{11}(4), 889--917 (2001)

\bibitem{Wan99}
Wang, X.: Convergence of {N}ewton's method and inverse function theorem in
  {B}anach space.
\newblock Math. Comp. \textbf{68}(225), 169--186 (1999)

\bibitem{Wan00}
Wang, X.: Convergence of {N}ewton's method and uniqueness of the solution of
  equations in {B}anach space.
\newblock IMA J. Numer. Anal. \textbf{20}(1), 123--134 (2000)

\bibitem{Wed73}
Wedin, P.: Perturbation theory for pseudo-inverses.
\newblock Nordisk Tidskr. Informationsbehandling (BIT) \textbf{13}, 217--232
  (1973)

\bibitem{Wom85}
Womersley, R.S.: Local properties of algorithms for minimizing nonsmooth
  composite functions.
\newblock Math. Programming \textbf{32}(1), 69--89 (1985)

\bibitem{WomFle86}
Womersley, R.S., Fletcher, R.: An algorithm for composite nonsmooth
  optimization problems.
\newblock J. Optim. Theory Appl. \textbf{48}(3), 493--523 (1986)

\bibitem{Xu09}
Xu, C.: Nonlinear least squares problems.
\newblock In: C.A. Floudas, P.M. Pardalos (eds.) Encyclopedia of Optimization,
  pp. 2626--2630. Springer US (2009)

\end{thebibliography}

\end{document}